\documentclass[11pt, a4paper, twoside,leqno]{amsart}
\usepackage[centering, totalwidth = 400pt, totalheight = 630pt]{geometry}
\usepackage{amssymb, amsmath, amsthm, enumerate, microtype, stmaryrd, url}
\usepackage[latin1]{inputenc}
\usepackage[arrow, matrix, tips, curve]{xy}
\usepackage{color}
\definecolor{darkgreen}{rgb}{0,0.45,0}
\usepackage[pagebackref,colorlinks,citecolor=darkgreen,linkcolor=darkgreen]{hyperref}
\SelectTips{cm}{10}

\newcommand{\cat}[1]{\mathbf{#1}}
\newcommand{\op}{\mathrm{op}}
\newcommand{\id}{\mathrm{id}}
\newcommand{\thg}{{\mathord{\text{--}}}}

\newcommand{\spn}[1]{{\langle{#1}\rangle}}

\newcommand{\defeq}{\mathrel{\mathop:}=}

\newcommand{\cd}[2][]{\vcenter{\hbox{\xymatrix#1{#2}}}}

\renewcommand{\phi}{\varphi}
\newcommand{\A}{{\mathcal A}}
\newcommand{\B}{{\mathcal B}}
\newcommand{\C}{{\mathcal C}}
\newcommand{\D}{{\mathcal D}}
\newcommand{\E}{{\mathcal E}}
\newcommand{\F}{{\mathcal F}}

\newcommand{\I}{{\mathcal I}}
\newcommand{\J}{{\mathcal J}}
\newcommand{\K}{{\mathcal K}}
\renewcommand{\L}{{\mathcal L}}

\renewcommand{\P}{{\mathcal P}}

\newcommand{\R}{{\mathcal R}}
\newcommand{\Ss}{{\mathcal S}}

\newcommand{\V}{{\mathcal V}}

\newcommand{\xtor}[1]{\cdl[@1]{{} \ar[r]|-{\object@{|}}^{#1} & {}}}

\makeatletter

\def\hookleftarrowfill@{\arrowfill@\leftarrow\relbar{\relbar\joinrel\rhook}}
\def\twoheadleftarrowfill@{\arrowfill@\twoheadleftarrow\relbar\relbar}
\def\leftbararrowfill@{\arrowdoublefill@{\leftarrow\mkern-5mu}\relbar\mapstochar\relbar\relbar}
\def\Leftbararrowfill@{\arrowdoublefill@{\Leftarrow\mkern-2mu}\Relbar\Mapstochar\Relbar\Relbar}
\def\leftringarrowfill@{\arrowdoublefill@{\leftarrow\mkern-3mu}\relbar{\mkern-3mu\circ\mkern-2mu}\relbar\relbar}
\def\lefttriarrowfill@{\arrowfill@{\mathrel\triangleleft\mkern0.5mu\joinrel\relbar}\relbar\relbar}
\def\Lefttriarrowfill@{\arrowfill@{\mathrel\triangleleft\mkern1mu\joinrel\Relbar}\Relbar\Relbar}

\def\hookrightarrowfill@{\arrowfill@{\lhook\joinrel\relbar}\relbar\rightarrow}
\def\twoheadrightarrowfill@{\arrowfill@\relbar\relbar\twoheadrightarrow}
\def\rightbararrowfill@{\arrowdoublefill@{\relbar\mkern-0.5mu}\relbar\mapstochar\relbar\rightarrow}
\def\Rightbararrowfill@{\arrowdoublefill@{\Relbar\mkern-2mu}\Relbar\Mapstochar\Relbar\Rightarrow}
\def\rightringarrowfill@{\arrowdoublefill@\relbar\relbar{\mkern-2mu\circ\mkern-3mu}\relbar{\mkern-3mu\rightarrow}}
\def\righttriarrowfill@{\arrowfill@\relbar\relbar{\relbar\joinrel\mkern0.5mu\mathrel\triangleright}}
\def\Righttriarrowfill@{\arrowfill@\Relbar\Relbar{\Relbar\joinrel\mkern1mu\mathrel\triangleright}}

\def\leftrightarrowfill@{\arrowfill@\leftarrow\relbar\rightarrow}
\def\mapstofill@{\arrowfill@{\mapstochar\relbar}\relbar\rightarrow}

\renewcommand*\xleftarrow[2][]{\ext@arrow 20{20}0\leftarrowfill@{#1}{#2}}
\providecommand*\xLeftarrow[2][]{\ext@arrow 60{22}0{\Leftarrowfill@}{#1}{#2}}
\providecommand*\xhookleftarrow[2][]{\ext@arrow 10{20}0\hookleftarrowfill@{#1}{#2}}
\providecommand*\xtwoheadleftarrow[2][]{\ext@arrow 60{20}0\twoheadleftarrowfill@{#1}{#2}}
\providecommand*\xleftbararrow[2][]{\ext@arrow 10{22}0\leftbararrowfill@{#1}{#2}}
\providecommand*\xLeftbararrow[2][]{\ext@arrow 50{24}0\Leftbararrowfill@{#1}{#2}}
\providecommand*\xleftringarrow[2][]{\ext@arrow 10{26}0\leftringarrowfill@{#1}{#2}}
\providecommand*\xlefttriarrow[2][]{\ext@arrow 80{24}0\lefttriarrowfill@{#1}{#2}}
\providecommand*\xLefttriarrow[2][]{\ext@arrow 80{24}0\Lefttriarrowfill@{#1}{#2}}

\renewcommand*\xrightarrow[2][]{\ext@arrow 01{20}0\rightarrowfill@{#1}{#2}}
\providecommand*\xRightarrow[2][]{\ext@arrow 04{22}0{\Rightarrowfill@}{#1}{#2}}
\providecommand*\xhookrightarrow[2][]{\ext@arrow 00{20}0\hookrightarrowfill@{#1}{#2}}
\providecommand*\xtwoheadrightarrow[2][]{\ext@arrow 03{20}0\twoheadrightarrowfill@{#1}{#2}}
\providecommand*\xrightbararrow[2][]{\ext@arrow 01{22}0\rightbararrowfill@{#1}{#2}}
\providecommand*\xRightbararrow[2][]{\ext@arrow 04{24}0\Rightbararrowfill@{#1}{#2}}
\providecommand*\xrightringarrow[2][]{\ext@arrow 01{26}0\rightringarrowfill@{#1}{#2}}
\providecommand*\xrighttriarrow[2][]{\ext@arrow 07{24}0\righttriarrowfill@{#1}{#2}}
\providecommand*\xRighttriarrow[2][]{\ext@arrow 07{24}0\Righttriarrowfill@{#1}{#2}}

\providecommand*\xmapsto[2][]{\ext@arrow 01{20}0\mapstofill@{#1}{#2}}
\providecommand*\xleftrightarrow[2][]{\ext@arrow 10{22}0\leftrightarrowfill@{#1}{#2}}
\providecommand*\xLeftrightarrow[2][]{\ext@arrow 10{27}0{\Leftrightarrowfill@}{#1}{#2}}

\makeatother

\newcommand{\twocong}[2][0.5]{\ar@{}[#2] \save ?(#1)*{\cong}\restore}
\newcommand{\twoeq}[2][0.5]{\ar@{}[#2] \save ?(#1)*{=}\restore}
\newcommand{\rtwocell}[3][0.5]{\ar@{}[#2] \ar@{=>}?(#1)+/l 0.2cm/;?(#1)+/r 0.2cm/^{#3}}
\newcommand{\ltwocell}[3][0.5]{\ar@{}[#2] \ar@{=>}?(#1)+/r 0.2cm/;?(#1)+/l 0.2cm/^{#3}}
\newcommand{\ltwocello}[3][0.5]{\ar@{}[#2] \ar@{=>}?(#1)+/r 0.2cm/;?(#1)+/l 0.2cm/_{#3}}
\newcommand{\dtwocell}[3][0.5]{\ar@{}[#2] \ar@{=>}?(#1)+/u  0.2cm/;?(#1)+/d 0.2cm/^{#3}}
\newcommand{\dltwocell}[3][0.5]{\ar@{}[#2] \ar@{=>}?(#1)+/ur  0.2cm/;?(#1)+/dl 0.2cm/^{#3}}
\newcommand{\drtwocell}[3][0.5]{\ar@{}[#2] \ar@{=>}?(#1)+/ul  0.2cm/;?(#1)+/dr 0.2cm/^{#3}}
\newcommand{\dthreecell}[3][0.5]{\ar@{}[#2] \ar@3{->}?(#1)+/u  0.2cm/;?(#1)+/d 0.2cm/^{#3}}
\newcommand{\utwocell}[3][0.5]{\ar@{}[#2] \ar@{=>}?(#1)+/d 0.2cm/;?(#1)+/u 0.2cm/_{#3}}
\newcommand{\dtwocelltarg}[3][0.5]{\ar@{}#2 \ar@{=>}?(#1)+/u  0.2cm/;?(#1)+/d 0.2cm/^{#3}}
\newcommand{\utwocelltarg}[3][0.5]{\ar@{}#2 \ar@{=>}?(#1)+/d  0.2cm/;?(#1)+/u 0.2cm/_{#3}}

\newcommand{\pushoutcorner}[1][dr]{\save*!/#1-1.2pc/#1:(-1,1)@^{|-}\restore}

\newdir{(}{{}*!<0em,-.14em>-\cir<.14em>{l^r}}
\newdir{ (}{{}*!/-5pt/\dir{(}}
\newdir{ >}{{}*!/-5pt/\dir{>}}

\theoremstyle{definition}

\theoremstyle{plain}
\newtheorem{Thm}{Theorem}
\newtheorem{Prop}[Thm]{Proposition}
\newtheorem{Cor}[Thm]{Corollary}
\newtheorem{Lemma}[Thm]{Lemma}

\numberwithin{equation}{section}

\theoremstyle{definition}

\newtheorem{Ex}[Thm]{Example}

\newtheorem{Rk}[Thm]{Remark}

\newcommand{\Lan}{\mathrm{Lan}}
\newcommand{\Ran}{\mathrm{Ran}}

\begin{document}
\leftmargini=2em
\title{Two-dimensional regularity and exactness}
\author{John Bourke}
\address{Department of Mathematics and Statistics, Masaryk University, Kotl\'a\v rsk\'a 2, Brno 60000, Czech Republic}
\email{bourkej@math.muni.cz} 
\author{Richard Garner}
\address{Department of Computing, Macquarie University, NSW 2109, Australia}
\email{richard.garner@mq.edu.au}

\date{\today}

\thanks{The first author acknowledges the support of the Grant agency of the Czech Republic, grant number
P201/12/G028.
The second author acknowledges the support of an Australian Research Council Discovery Project, grant number DP110102360.}

\begin{abstract}
We define notions of regularity and (Barr-)exactness for $2$-categories. In fact, we define \emph{three} notions of regularity and exactness, each based on one of the three canonical ways of factorising a functor in $\cat{Cat}$: as (surjective on objects, injective on objects and fully faithful),
as (bijective on objects, fully faithful), and as
(bijective on objects and full, faithful). The correctness of our notions is justified using the theory of lex colimits~\cite{Garner2011Lex-colimits} introduced by Lack and the second author.
Along the way, we develop an abstract theory of regularity and exactness relative to a kernel--quotient factorisation, extending earlier work of Street and others~\cite{Street1982Exact,Betti1995Factorizations}.
\end{abstract}

\maketitle
\section{Introduction}\label{sec:introduction}
This paper is concerned with two-dimensional generalisations of the notions of \emph{regular} and \emph{Barr-exact} category. There are in fact a plurality of such generalisations and a corresponding plurality of articles exploring these generalisations---see~\cite{Benabou1989Some,Betti1995Factorizations,Bourn20102-categories,Carboni1994Modulated,Dupont2008Abelian,Dupont2003Proper,Street1982Characterizations,Street1982Two-dimensional}, for example. Examining this body of work, one finds a clear consensus as to the form such generalised notions should take: one considers a $2$-category or bicategory equipped with finite limits and with a certain class of colimits, and requires certain ``exactness'' conditions to hold between the finite limits and the specified colimits. For generalised regularity, the finite limits are used to form the ``kernel'' of an arrow; the specified colimits are just those needed to form ``quotients'' of such kernels; and the exactness conditions ensure that the process of factoring an arrow through the quotient of its kernel gives rise to a well-behaved factorisation system on the $2$-category or bicategory in question. For generalised Barr-exactness, the finite limits are used to specify ``congruences'' (the maximal finite-limit structure of which all ``kernels'' are instances); the specified colimits are those required to form quotients of congruences; and the exactness conditions extend those for regularity by demanding that every congruence be ``effective''---the kernel of its quotient.

Whilst this general schema for regularity and Barr-exactness notions is clear enough (and as we shall soon see, makes sense in a more general setting than just that of $2$-categories), the details in individual cases are less so. The main complication lies in ascertaining the right exactness conditions to impose between the finite limits and the specified colimits. Previous authors have done so in an essentially ad hoc manner, guided by intuition and a careful balancing of the opposing constraints of sufficient examples and sufficient theorems. The first contribution of this paper is to show that any instance of the schema admits a canonical, well-justified choice of exactness conditions, which automatically implies many of the desirable properties that a generalised regular or Barr-exact category should have.

 We obtain this canonical choice from the theory of \emph{lex colimits} developed by the second author and Lack in~\cite{Garner2011Lex-colimits}. This is a framework for dealing with $\V$-categorical structures involving limits, colimits, and exactness between the two;
one of the key insights is that, for a given class of colimits, the appropriate exactness conditions to impose are just those which hold between finite limits and the given colimits in the base $\V$-category $\V$; more generally, in any ``$\V$-topos'' (lex-reflective subcategory of a presheaf $\V$-category). Applied in the case $\V = \cat{Set}$, this theory justifies the exactness conditions for the notions of regular and Barr-exact category as well as those of extensive, coherent or adhesive categories; applied in the case $\V = \cat{Cat}$, it will \emph{provide} us with the exactness conditions for our generalised regularity and Barr-exactness notions.

The second contribution of this paper is to study in detail three particular notions of two-dimensional regularity and Barr-exactness. As we have said, there are a range of such notions; in fact, there is one for each well-behaved orthogonal factorisation system on $\cat{Cat}$, and the three examples we consider arise from the following ways of factorising a functor:
\begin{enumerate}[(i)]
\item (surjective on objects, injective on objects and fully faithful);
\item (bijective on objects, fully faithful); and
\item (bijective on objects and full, faithful).
\end{enumerate}
Of course, many other choices are possible---interesting ones for further investigation would be (final, discrete opfibration)~\cite{Street1973The-comprehensive} and (strong liberal, conservative)~\cite{Carboni1994Modulated}---but amongst all possible choices, these three are the most evident and in some sense the most fundamental. 
For (i), the notion of regularity we obtain is more or less that defined in~\cite[\S1.19]{Street1982Two-dimensional}; the exactness conditions amount simply to the stability under pullback of the quotient morphisms. In the case (ii), we obtain the folklore construction of (bijective on objects, fully faithful) factorisations via the codescent object of a higher kernel; see~\cite[\S3]{Street2004Categorical}, for example. However, the exactness conditions required do \emph{not} simply amount to stability under pullback of codescent morphisms; one must also impose the extra condition that, if $A \to B$ is a codescent morphism, then so also is the diagonal map $A \to A \times_B A$. This condition, forced by the general theory of~\cite{Garner2011Lex-colimits}, has not been noted previously and is moreover, substantive: for example, the category $\cat{Set}$, seen as a locally discrete $2$-category, satisfies all the other prerequisites for regularity in this sense, but \emph{not}  this final condition. Finally, the regularity notion associated with the factorisation system (iii) appears to be new, although an abelian version of it is considered in~\cite{Kasangian2000Factorization}. The corresponding analogues of Barr-exactness for (i), (ii) and (iii) supplement the regularity notions by requiring effective quotients of appropriate kinds of congruences: for (i), these are the congruences discussed in~\cite[\S1.8]{Street1982Two-dimensional}; for (ii) they are the~\emph{cateads} of~\cite{Bourn20102-categories}; whilst for (iii), they are internal analogues of the notion of category equipped with an equivalence relation on each hom-set, compatible with composition in each variable.

%
We find that there are many $2$-categories which are regular or exact in the senses we define. $\cat{Cat}$ is so essentially by definition; and this implies the same result for any presheaf $2$-category $[\C^\op, \cat{Cat}]$. The category of algebras for any $2$-monad on $\cat{Cat}$ which is \emph{strongly finitary} in the sense of~\cite{Kelly1993Finite-product-preserving} is again regular and exact in all senses; which encompasses such examples as the $2$-category of monoidal categories and strict monoidal functors; the $2$-category of categories equipped with a monad; the $2$-category of categories with finite products and strict product-preserving functors; and so on. Another source of examples comes from internal category theory. If $\E$ is a category with finite limits, then $\cat{Cat}(\E)$ is \emph{always} regular and exact relative to the factorisation system (ii); if $\E$ is moreover regular or Barr-exact in the usual $1$-categorical sense, then $\cat{Cat}(\E)$ will be regular or exact relative to (i) and (iii) also. Finally, we may combine the above examples in various ways: thus, for instance, the $2$-category of internal monoidal categories in any Barr-exact category $\E$ is regular and exact in all three senses. 

As we mentioned in passing above, there is nothing inherently two-dimensional about the schema for generalised regularity and exactness; it therefore seems appropriate to work---at least initially---in a more general setting. Over an arbitrary enrichment base $\V$, one may define a notion of \emph{kernel--quotient system} whose basic datum is a small $\V$-category $\F$ describing the shape of an ``exact fork'': the motivating example takes $\V = \cat{Set}$ and $\F = \bullet \rightrightarrows \bullet \rightarrow \bullet$. Given only this $\F$, one may define analogues of all the basic constituents of the theory of regular and Barr-exact categories; the particular examples of interest to us will arise from three suitable choices of $\F$ in the case $\V = \cat{Cat}$. The theory of kernel--quotient systems was first investigated by Street in unpublished work~\cite{Street1982Exact}, and developed further in a preprint of Betti and Schumacher~\cite{Betti1995Factorizations}; a published account of some of their work may be found in~\cite{Dupont2008Abelian}. As indicated above, the new element we bring is the use of the ideas of~\cite{Garner2011Lex-colimits} to justify the exactness conditions appearing in the notions of $\F$-regularity and $\F$-exactness.

Finally, let us remark on what we do \emph{not} do in this paper. All the two-dimensional exactness notions we consider will be strict $2$-categorical ones; thus we work with $2$-categories rather than bicategories, $2$-functors rather than homomorphisms, weighted $2$-limits rather than bilimits, and so on. In other words, we are working within the context of $\cat{Cat}$-enriched category theory; this allows us to apply the theory of~\cite{Garner2011Lex-colimits} directly, and noticeably simplifies various other aspects of our investigations. There are bicategorical analogues of our results, which are conceptually no more difficult but are more technically involved; we have therefore chosen to present the $2$-categorical case here, reserving the bicategorical analogue for future work.

We now describe the contents of this paper. We begin in Section~\ref{sec:kerquot} by defining kernel--quotient systems and developing aspects of their theory; as explained above, this material draws on~\cite{Street1982Exact} and~\cite{Betti1995Factorizations}. We do not yet define the notions of regularity and exactness relative to a kernel--quotient system $\F$; before doing so, we must recall, in Section~\ref{sec:phi-lex-cocompleteness}, the relevant aspects of the lex colimits of~\cite{Garner2011Lex-colimits}. This then allows us, in Section~\ref{sec:freg}, to complete the definitions of $\F$-regularity and $\F$-exactness, and to show that many of the desirable properties of an $\F$-regular or $\F$-exact category follow already at this level of generality.

This completes the first main objective of the paper; in Section~\ref{sec:2dkerquot}, we commence the second, by introducing the 
 two-dimensional kernel--quotient systems corresponding to (i)--(iii) above, and studying their properties. In Section~\ref{sec:elementary}, we describe in elementary terms the notions of two-dimensional regularity and exactness associated to these systems, using the two-dimensional sheaf theory of~\cite{Street1982Two-dimensional}; in the penultimate Section~\ref{sec:relationships}, we consider the interrelationships between these notions; and finally, in Section~\ref{sec:examples}, we describe in more detail the range of examples outlined above.

\section{Kernel--quotient systems}\label{sec:kerquot}
In this section and the following two, we work in the context of the $\V$-category theory of~\cite{Kelly1982Basic}, for $\V$ some locally finitely presentable symmetric monoidal closed category; in the final four sections, we will specialise to the case $\V = \cat{Cat}$. Before starting our exposition proper, let us recall the notion of $\V$-orthogonality: a map $f \colon A \to B$ of a $\V$-category $\C$ is said to be \emph{$\V$-orthogonal} to $g \colon C \to D$---written $f \mathbin \bot g$---if the square
\begin{equation}\label{eq:orthogonality-square}
\cd[@-0.5em]{
  \C(B,C) \ar[r]^{\C(B,g)} \ar[d]_{\C(f,C)} &
  \C(B,D) \ar[d]^{\C(f,D)} \\
  \C(A,C) \ar[r]_{\C(A,g)} & \C(A,D)
}
\end{equation}
is a pullback in $\V$. A map $f$ is orthogonal to an object $C$, written $f \mathbin \bot C$, if $\C(f, C) \colon \C(B,C) \to \C(A,C)$ is invertible, and similarly $A \mathbin \bot g$ if $\C(A, g)$ is invertible.
\begin{Lemma}\label{lem:orth}
\begin{enumerate}[(i)]
\item $f \colon A \to B$ is invertible if and only if $f \mathbin \bot C$ for all $C \in \C$;
\item If $L \dashv R \colon \C \to \D$ then $Lf \mathbin \bot g$ if and only if $f \mathbin \bot Rg$;
\item Given $f \colon A \to B$ and $g \colon C \to D$ in $\C$, we have $f \mathbin \bot g$ in $\C$ if and only if the object $f$ is orthogonal to the map $(1_C, g) \colon 1_C \to g$ in $[\mathbf 2, \C]$.
\end{enumerate}
\end{Lemma}
\begin{proof}
Only (iii) is non-trivial. Note that in~\eqref{eq:orthogonality-square}, the pullback of $\C(A,g)$ and $\C(f,C)$ is the hom-object $[\mathbf 2, \C](f, g)$, whilst $\C(B,C)$ is isomorphic to $[\mathbf 2, \C](f, 1_C)$; in these terms, the induced comparison map is given by postcomposing with $(1_C, g)$. Thus to say that $f \mathbin \bot g$ is equally to say that $f \mathbin 
\bot (1_C, g)$ in $[\mathbf 2, \C]$.
\end{proof}

We now turn the main object of study of this section: a notion of \emph{kernel--quotient system} which captures the abstract properties of the kernel-pair--coequaliser construction central to the notions of regular and Barr-exact category.  
 As noted in the introduction, the material of this section is based on~\cite{Street1982Exact,Betti1995Factorizations}.
 
The basic data for a kernel--quotient system is a finitely presentable $\V$-category $\F$ that contains as a full subcategory $\cat 2$, the free $\V$-category on an arrow $1 \rightarrow 0$\footnote{In fact, we can weaken the requirement of finite presentability of $\F$; we really only need that each hom-object $\F(x,1)$ and $\F(x,0)$ be finitely presentable in $\V$, which is what is needed to ensure that right Kan extension along $J \colon \mathbf 2 \to \F$ can be computed using finite limits.}. Given such an $\F$, we write $\K$ for the full subcategory of $\F$ on all objects except $0$ and write 
$I \colon \K \rightarrow \F \leftarrow  \mathbf 2 \colon J$ for the induced pair of full inclusions. We think of the category $\K$ as the shape of ``kernel-data'' and the category $\F$ as the shape of an ``exact fork'', with the functors $I$ and $J$ indicating how a kernel and a quotient morphism sit inside such a fork.
The motivating case is that corresponding to the one-dimensional regular factorisation: we take $\V = \cat{Set}$, and $\F$ to be the category generated by the graph $2 \rightrightarrows 1 \rightarrow 0$ subject to the relation identifying the two composites $2 \rightrightarrows 0$. Another basic example is that which underlies abelian categories: we take $\V = \cat{Ab}$ and $\F$ the $\cat{Ab}$-category generated by the graph $2 \to 1 \to 0$ subject to the relation that the composite $2 \to 0$ is the zero map.

Given an $\F$ of this kind, we obtain for any sufficiently complete and cocomplete $\V$-category $\C$ an adjunction
between morphisms in $\C$ and kernel-data in $\C$ as on the left in
\begin{equation*}
\cd{
 [\mathbf 2, \C] \ar@<-6pt>[r]_{\Ran_J} \ar@{}[r]|{\bot} &
 [\F, \C] \ar@<-6pt>[r]_{I^\ast} \ar@<-6pt>[l]_{J^\ast} \ar@{}[r]|{\bot} &
 [\K, \C] \ar@<-6pt>[l]_{\Lan_I}
}
\qquad \qquad 
\cd{
 [\mathbf 2, \C] \ar@<-6pt>[r]_{K} \ar@{}[r]|{\bot} &
 [\K,\C]' \ar@<-6pt>[l]_{Q}
}\rlap{ .}
\end{equation*}
In practice, though we will always assume that $\C$ is finitely complete---which suffices to assure the existence of the right adjoint $K$---we will not assume the existence of all colimits necessary to construct the left adjoint. Nonetheless, if we write $[\K, \C]' \subset [\K, \C]$ for the full subcategory of objects $X$ for which $\Lan_I X$ exists, then we obtain a $\V$-functor $Q \colon [\K, \C]' \to [\mathbf 2, \C]$ which is left adjoint to $K$ insofar as it is defined. 
In particular, if the left adjoint exists at every $X \in [\K, \C]$ which is an $\F$-\emph{kernel}---that is, in the image of $K$---then we obtain an adjunction as on the right above, and say that $\C$ \emph{admits the kernel--quotient adjunction for $\F$}.

Now  we define a morphism $f \colon A \to B$ in the finitely complete $\C$ to be:
\begin{itemize}
\item \emph{$\F$-monic} if the morphism $K(1_A,f) \colon K(1_A) \to  K(f)$ is an isomorphism;
\item \emph{$\F$-strong epi} if $f \mathbin \bot g$ for every $\F$-monic $g \colon C \to D$;
\item an \emph{$\F$-quotient map} if it lies in the essential image of $Q$;
\item an \emph{effective $\F$-quotient map} if $QKf$ exists and the counit map $QKf \to f$ is invertible; equivalently, if the identity map $Kf \to Kf$ exhibits $f$ as $QKf$.
\end{itemize}




In the motivating one-dimensional example, the $\F$-monics and $\F$-strong epis are the monics and the strong epis, whilst the $\F$-quotient maps and  effective $\F$-quotient maps are the regular epis; the following proposition generalises some well-known properties of these classes to the case of a general $\F$.
\begin{Prop}\label{prop:f-map-props}
Let $\C$ be a finitely complete $\V$-category.
\begin{enumerate}[(a)]
\item $g \colon C \to D$ is $\F$-monic in $\C$ if and only if $\C(X, g)$ is $\F$-monic in $\V$ for all $X \in \C$.
\item Every $\F$-quotient map in $\C$ is an $\F$-strong epi.
\item If $\C$ admits the kernel--quotient adjunction for $\F$, then $g \colon C \to D$ is $\F$-monic if and only $f \mathbin \bot g$ for every $\F$-quotient  map $f$.
\item $\F$-strong epis in $\C$ are closed under composition and identities, under pushout along arbitrary morphisms, and under colimits in $\C^\mathbf 2$; moreover, if $h = gf$ is $\F$-strong epi, and $f$ is either $\F$-strong epi or epi, then $g$ is $\F$-strong epi.
\item $\F$-monics in $\C$ are closed under composition and identities, under pullback along arbitrary morphisms, and under limits in $\C^\mathbf 2$; moreover if $h = gf$ is $\F$-monic, and $g$ is either $\F$-monic or monic, then $f$ is $\F$-monic.
\item Any $\V$-functor preserving finite limits preserves $\F$-monics; any left adjoint $\V$-functor preserves $\F$-strong epis and $\F$-quotient maps, and will also preserve effective $\F$-quotient maps so long as it preserves finite limits.
\end{enumerate}
\end{Prop}
\begin{proof}
For (a), each representable $\C(X,\thg)$ preserves limits, and so commutes with the formation of kernels; the result is now immediate by the Yoneda lemma.

For (b), let $QX$ be an $\F$-quotient map in $\C$. For any $\F$-monic $g 
\colon C \to D$, we have $K(1_C, g)$ invertible, whence $X \mathbin \bot K(1_C, g)$ in $[\K, \C]$; whence $QX \mathbin \bot (1_C, g)$ in $[\mathbf 2, \C]$; whence $QX \mathbin \bot g$ in $\C$. So $QX$ is $\F$-strong epi.
%

For (c),  $g \colon C \to D$ is $\F$-monic iff $K(1_C, g)$ is invertible, iff $X \mathbin \bot K(1_C, g)$ for all $X \in [\K, \C]'$, iff $QX \mathbin \bot (1_C,g)$ for all $X \in [\K, \C]'$, iff $QX \mathbin \bot g$ for all $X \in [\K, \C]'$.
%

Part (d) follows from the definition of $\F$-strong epis by an orthogonality property in $\C$; as for (e), observe that, since $\F$-monics and all the listed constructions are preserved and jointly reflected by the representables $\C(X, \thg)$, it suffices to prove the case $\C = \V$; and this follows from the orthogonality characterisation of $\F$-monics in $\V$ given in (c).

Finally, for (f), 
 any $\V$-functor preserving finite limits commutes with the construction of kernels and so preserves $\F$-monics. In particular, any right adjoint preserves $\F$-monics, whence by orthogonality, any left adjoint preserves $\F$-strong epis. Furthermore, any left adjoint $\V$-functor commutes with the construction of quotients $Q$, and so will preserve $\F$-quotient maps; if it also preserves finite limits, then it commutes with the construction of kernels, and so preserves effective $\F$-quotients.
\end{proof}

One point which distinguishes the general kernel--quotient system from the motivating kernel-pair--coequaliser system concerns the distinction between quotient maps and effective quotient maps. In the motivating case, \emph{every} $\F$-quotient map is effective, which is to say that every regular epimorphism is the coequaliser of its own kernel-pair. For the general kernel--quotient system this need not be the case; a counterexample is given in Proposition~\ref{prop:codescent-not-effective} below. However, we do have the following result:
\begin{Prop}\label{prop:effectivity-ker-quot}
For any $\C$ which admits the kernel--quotient adjunction for $\F$, the 
following are equivalent:
\begin{enumerate}[(a)]
\item Every $\F$-quotient map in $\C$ is effective;
\item Every $\F$-quotient of an $\F$-kernel in $\C$ is effective;
\item Every $\F$-kernel in $\C$ is effective.
\end{enumerate}
\end{Prop}
In the statement of this result, we call $X \in [\K, \C]'$  \emph{effective} if the unit $X \to KQX$ of the kernel--quotient adjunction at $X$ is invertible.
\begin{proof}
As before, we have the adjunction $Q \dashv K$ on the left below, and this restricts to an adjunction as on the right; here we write $\cat{Ker}(\C) \subset [\K, \C]'$ for the full sub-$\V$-category spanned by the $\F$-kernels.
\begin{equation*}
\cd{
 [\mathbf 2, \C] \ar@<-6pt>[r]_{K} \ar@{}[r]|{\bot} &
 [\K,\C]' \ar@<-6pt>[l]_{Q}
}\qquad \qquad \quad
\cd{
 [\mathbf 2, \C] \ar@<-6pt>[r]_{K} \ar@{}[r]|{\bot} &
 \mathbf{Ker}(\C)\rlap{ .} \ar@<-6pt>[l]_{Q}
}
\end{equation*}
Condition (a) says that the whiskered counit $\epsilon Q \colon QKQ \Rightarrow Q$ of the left-hand adjunction is invertible; (b) that the corresponding $\epsilon Q$ for the right-hand adjunction is invertible; and (c) that the unit $\eta K \colon K \Rightarrow KQK$ of either adjunction is invertible.
The equivalence of (a) and (c) is now a standard fact about adjunctions; likewise that of (b) and (c).
%
%
\end{proof}
%
%
%
%
%

Let us next see how kernel--quotient systems give rise to factorisation systems. Let $\C$ be a $\V$-category admitting the kernel--quotient adjunction for $\F$. Observe that since $I \colon \K \to \F$ and $J \colon \mathbf 2 \to \F$ are injective on objects and fully faithful, the functors $\Lan_I$ and $\Ran_J$, insofar as they are defined, may be taken to be strict sections of $I^\ast$ and $J^\ast$ respectively; whence the kernel--quotient adjunction $Q \dashv K \colon [\mathbf 2, \C] \to [\K, \C]'$ may be taken so as to commute strictly with the functors $[\mathbf 2, \C] \to \C$ and $[\K, \C]' \to \C$ given by evaluation at the object~$1$. Consequently, the counit of this adjunction at $f \in [\mathbf 2, \C]$ may be taken to be of the form
\begin{equation}\label{eq:factorisation}
\cd[@-0.3em]{
A \ar[r]^{1_A} \ar[d]_{QKf} & A \ar[d]^f \\
\bullet \ar[r]_{\epsilon_f} & B\rlap{ .}
}
\end{equation}
We thus have a factorisation $f = \epsilon_f \circ QKf$ of each map of $\C$. The first factor $QKf$ is always an $\F$-quotient map; if the second factor $\epsilon_f$ is always an $\F$-monic, we shall say that \emph{$\F$-kernel--quotient factorisations in $\C$ converge immediately}. 

\begin{Prop}\label{prop:fkera}
If $\F$-kernel--quotient factorisations converge immediately in $\C$, then it admits an ($\F$-quotient, $\F$-monic) factorisation system, and the classes of $\F$-strong epis, $\F$-quotients and effective $\F$-quotients coincide.
\end{Prop}
\begin{proof}
If kernel--quotient factorisation converge immediately in $\C$, then every map admits an ($\F$-quotient, $\F$-monic) factorisation; since these two classes of maps are orthogonal, they must therefore comprise the two classes of a factorisation system. It remains to show that every $\F$-strong epi $f$ is an effective $\F$-quotient. Now both $f = 1 \circ f$ and $f = \epsilon_f \circ QKf$ are ($\F$-strong epi, $\F$-monic) factorisations of $f$, whence by the essential-uniqueness of such factorisations, $\epsilon_f$ is invertible, which is to say that $f$ is an effective $\F$-quotient as required.
\end{proof}
\begin{Rk}\label{rk:not-immediate}
It is of course possible for $\F$-kernel--quotient factorisations to exist in some $\C$ without converging immediately. For example, take $\V = \cat{Set}$, $\F$ to be the kernel-pair--coequaliser system, and $\C = \cat{Cat}$, and consider the functor
\[
\cd{ a \ar[r] \ar[d] & b & b' \ar[d] \\
c & c' \ar[r] & d} \qquad \longrightarrow \qquad
\cd{ a \ar[r] \ar[d] & b \ar[d] \\ c \ar[r] & d}
\]
into the generic commuting square $\mathbf 2 \times \mathbf 2$ which identifies $b$ with $b'$ and $c$ with $c'$. The kernel-pair--coequaliser factorisation of this functor maps through the generic \emph{non-commuting} square; the second half of this factorisation is clearly not monic and so kernel-pair--coequaliser factorisations do not converge immediately in $\cat{Cat}$. What \emph{is} true for this $\F$ is that kernel-pair--quotient factorisations converge immediately in any regular category; a suitable generalisation of this fact will be given in Proposition~\ref{prop:ker-quot-immediate-fregular} below.
\end{Rk}

We now describe how a kernel--quotient system $\F$ gives rise to an associated notion of $\F$-congruence in each finitely complete $\V$-category $\C$. 
First, let us define an $\F$-\emph{congruence axiom} to be a map $h \colon \phi \to \psi$ between finitely presentable objects in $[\K, \V]$ 
such that $h \mathbin \bot X$ for every $\F$-kernel $X \in [\K, \V]$. Now an \emph{$\F$-congruence} in $\V$ is an object $X \in [\K, \V]$ such that $h \mathbin \bot X$ for every  every $\F$-congruence axiom $h$; more generally, 
an \emph{$\F$-congruence} in a finitely complete category $\C$ is an object $X \in [\K, \C]$ such that, for each $\F$-congruence axiom $h$, the morphism $\{h, X\} \colon \{\psi, X\} \to \{\phi, X\}$ between weighted limits is invertible. By the Yoneda lemma and the representability of limits, $X$ is an $\F$-congruence in $\C$ if and only if $\C(A, X)$ is one in $\V$ for each $A \in \C$. 

\begin{Prop}\label{prop:f-ker-is-f-cong}
Every $\F$-kernel is an $\F$-congruence.
\end{Prop}
\begin{proof}
The result is clearly true in $\V$, whilst if $X$ is the $\F$-kernel of $f$ in $\C$, then for each $A \in \C$ we have $\C(A, X)$ the $\F$-kernel of $\C(A,f)$ in $\V$ and so an $\F$-congruence.
\end{proof}

In practice, it will be convenient to describe $\F$-congruences in terms of a generating set $\Ss$ of $\F$-congruence axioms; here, we call a set $\Ss$ \emph{generating} if $\V$-orthogonality of $X \in [\K, \V]$ to all congruence axioms in $\Ss$ implies $\V$-orthogonality of $X$ to \emph{every} congruence axiom. Observe that as there are only a small set of isomorphism-classes of arrows $\phi \to \psi$ between finitely presentable objects in $[\K, \V]$, every kernel--quotient system admits a \emph{small} generating set; we use this fact in the proof of Proposition~\ref{prop:ker-quot-lex-weights} below.

\begin{Prop}\label{prop:f-congruence-generating}
If $\Ss$ is a generating set of $\F$-congruence axioms, then an object $X \in [\K, \C]$ is an $\F$-congruence if and only if $\{h, X\}$ is invertible for every $h \in \Ss$. 
\end{Prop}
\begin{proof}
$X \in [\K, \C]$ is an $\F$-congruence if and only if each $\C(A, X)$ is an $\F$-congruence in $\V$; if and only if each $\C(A, X)$ is orthogonal to all maps in $\Ss$; if and only if $\{h, X\}$ is invertible for all $h \in \Ss$.
\end{proof}

\begin{Prop}\label{prop:f-effective-cong}
If $\Ss$ is a set of $\F$-congruence axioms such that every $X \in [\K, \V]$ orthogonal to $\Ss$ is effective, then $\Ss$ is a generating set and every $\F$-congruence in $\V$ is effective. 
\end{Prop}
\begin{proof}
If every $X \in [\K, \V]$ orthogonal to $\Ss$ is effective, then every such $X$ is an $\F$-kernel, and hence by definition orthogonal to every $\F$-congruence axiom. Thus $\Ss$ is a generating set, and our hypothesis says that every $\F$-congruence in $\V$ is effective.
\end{proof}

\begin{Ex}
Consider the motivating kernel-pair--coequaliser system. The category of kernel-data in $\cat{Set}$ is the category of directed graphs $\cat{Set}^{\mathbb P}$; in which we have congruence axioms
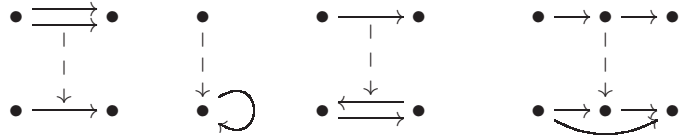
\begin{equation*}
\cd{
\bullet \ar@<3pt>[r] \ar@<-3pt>[r]_{}="a" & \bullet  \\
\bullet \ar[r]^{}="b" & \bullet
\ar@{-->} "a";"b"
}
\qquad \cd{
\bullet \ar@{-->}[d] \\
\bullet \ar@(ur,dr)[]
} \qquad \quad
\cd{
\bullet \ar[r]_{}="a" & \bullet  \\
\bullet \ar@<-3pt>[r] \ar@{<-}@<3pt>[r]^{}="b" & \bullet
\ar@{-->} "a";"b"
}\qquad \quad
\cd[@C-1em]{
\bullet \ar[r] & \bullet \ar@{-->}[d] \ar[r] & \bullet \\
\bullet \ar@/_9pt/[rr] \ar[r] & \bullet \ar[r] & \bullet
}
\end{equation*}
\vskip0.5\baselineskip \noindent
expressing that every kernel-pair $A \times_B A \rightrightarrows A$ is a binary relation which is reflexive, symmetric and transitive; i.e., an equivalence relation. Since
equivalence relations in $\cat{Set}$ are effective, it follows from the preceding two propositions that the $\F$-congruences in any finitely complete $\C$ are the equivalence relations.
\end{Ex}

Recall that a finite-limit preserving functor between regular (respectively, Barr-exact) categories preserves coequalisers of kernel-pairs (respectively, equivalence relations) if and only if it preserves regular epimorphisms. Our final result in this section generalises this result to an arbitrary kernel--quotient factorisation system.

\begin{Prop}\label{prop:preserve-quotients-preserve-colims}
Let $\C$ and $\D$ be finitely complete categories and $F \colon \C \to \D$ a finite-limit-preserving functor.
\begin{enumerate}[(a)]
\item If $\F$-kernels in $\C$ and $\D$ admit $\F$-quotients and are effective, then $F$ preserves $\F$-quotients of $\F$-kernels if and only if it preserves $\F$-quotient morphisms.
\vskip0.5\baselineskip
\item If $\F$-congruences in $\C$ and $\D$ admit $\F$-quotients and are effective, then $F$ preserves $\F$-quotients of $\F$-congruences if and only if it preserves $\F$-quotient morphisms.
\end{enumerate}
\end{Prop}
\begin{proof}
We first prove (a). Suppose first that $F$ preserves $\F$-quotients of $\F$-kernels. Given an $\F$-quotient morphism $f$, we have by effectivity that it is the $\F$-quotient of its own $\F$-kernel $Kf$; whence $Ff$ is the $\F$-quotient of $F(Kf)$, and in particular, an $\F$-quotient morphism.
Conversely, suppose that $F$ preserves $\F$-quotient morphisms. Given a morphism $f \in \C$, we form its kernel $Kf$ and the cocone exhibiting $QKf$ as the $\F$-quotient of $Kf$; we must show that the image under $F$ of this cocone exhibits $F(QKf)$ as the $\F$-quotient of $F(Kf)$. Now $QKf$ is an $\F$-quotient morphism; hence so too is $F(QKf)$, and so by effectivity in $\D$, must be the $\F$-quotient of its own $\F$-kernel $K(F(QKf))$. But $K(F(QKf)) \cong F(K(QKf)) \cong F(Kf)$, since $F$ preserves $\F$-kernels and all kernels in $\C$ are effective, and so $F(QKf)$ is the $\F$-quotient of $F(Kf)$ as required.

As for (b), if $\F$-congruences are effective in $\C$ and $\D$, then every $\F$-congruence is in fact an $\F$-kernel (the $\F$-kernel of its own $\F$-quotient).
Since $\F$-kernels are $\F$-congruences by Proposition~\ref{prop:f-ker-is-f-cong}, the result now follows from (a).
\end{proof}

\section{Revision of lex colimits}\label{sec:phi-lex-cocompleteness}
Our objective now is to define analogues of regularity and (Barr-)exactness with respect to a given kernel--quotient system $\F$. A finitely complete $\V$-category will be \emph{$\F$-regular} if it admits $\F$-quotients of $\F$-kernels and these behave well with respect to finite limits, in the sense of interacting with them in the same way as in the base $\V$-category $\V$. Similarly, a finitely complete $\V$-category will be called \emph{$\F$-exact} if it admits $\F$-quotients of $\F$-congruences which behave well with respect to finite limits.
As explained in the introduction, we shall give precise form to the good behaviour that is expected to hold using the theory of \emph{lex colimits} developed in~\cite{Garner2011Lex-colimits}. In this section we revise the necessary results from that theory.

By a \emph{class of weights for lex colimits}, or more briefly, a \emph{class of lex-weights}~\cite[Section~3]{Garner2011Lex-colimits}, we mean a collection $\Phi$ of $\V$-functors $\{\phi \colon \I^\op \to \V\}$ where the domain of each $\phi \in \Phi$ is a small and finitely complete $\V$-category (note that the functors $\phi$ are not expected to preserve finite limits). A finitely complete $\V$-category $\C$ is said to be \emph{$\Phi$-lex-cocomplete} when for every $\phi \colon \I^\op \to \V$ in $\Phi$ and every finite-limit-preserving $\V$-functor $D \colon \I \to \C$, the colimit $\phi \star D$ exists in $\C$. In the following section, we shall exhibit, for any kernel--quotient system $\F$, classes of lex-weights $\Phi_{\F}^{\mathrm{reg}}$ and $\Phi_\F^\mathrm{ex}$ such that a category is $\Phi_{\F}^{\mathrm{reg}}$- or $\Phi_\F^\mathrm{ex}$-lex-cocomplete just when  it admits $\F$-quotients of $\F$-kernels, respectively $\F$-congruences.

We now arrive at the crucial notion of \emph{$\Phi$-exact} category with respect to a class of lex-weights $\Phi$; this is a $\V$-category with finite limits and $\Phi$-lex-colimits in which the $\Phi$-lex-colimits are ``well behaved'' in the above sense with respect to finite limits. Applying this to the classes of lex-weights $\Phi_\F^\mathrm{reg}$ and $\Phi_\F^\mathrm{ex}$ associated to a kernel--quotient system $\F$ will yield our notions of $\F$-regular and $\F$-exact $\V$-category.

As explained in~\cite{Garner2011Lex-colimits}, the key to describing the nature of $\Phi$-exactness is the construction of the \emph{$\Phi$-exact completion} of a finitely complete $\V$-category. Recall first that any $\V$-category $\C$ has a \emph{free cocompletion} $\P \C$, whose objects are those $F \colon \C^\op \to \V$ which are small colimits of representables, and whose hom-objects are given by the usual end formula; see~\cite{Kelly1982Basic}, for example. $\P \C$ is always cocomplete, and the results of~\cite{Day2007Limits} show that it is moreover finitely complete whenever $\C$ is so. For a finitely complete $\C$, we now construct its $\Phi$-exact completion $\Phi(\C)$ as the smallest full, replete, sub-$\V$-category of $\P \C$ which contains the representables, contains the limit of any finitely-weighted diagram whenever it contains the diagram, and for any $\phi \in \Phi$, contains the colimit $\phi \star D$ of a finite-limit-preserving $D \colon \I \to [\C^\op, \V]$ whenever it contains each $DI$. 

The assignation $\C \mapsto \Phi(\C)$ is the action on objects of a pseudomonad $\Phi$ on the $2$-category $\V\text-\cat{LEX}$ of finitely complete $\V$-categories,  finite-limit preserving $\V$-functors, and $\V$-natural transformations; and we call a finitely complete $\V$-category \emph{$\Phi$-exact}~\cite[Proposition~3.4]{Garner2011Lex-colimits} when it admits pseudoalgebra structure for this pseudomonad. The pseudomonad $\Phi$ is of the kind which is sometimes called \emph{Kock-Z\"oberlein}~\cite{Kock1995Monads}, so that for the finitely complete $\C$ to admit pseudoalgebra structure is equally well for the restricted Yoneda embedding $\C \to \Phi(\C)$ to admit a finite-limit-preserving left adjoint. This implies, in particular, that a $\C$ admitting pseudoalgebra structure is $\Phi$-lex-cocomplete; however, the requirement that the left adjoint preserve finite limits forces the additional  compatibilities between finite limits and $\Phi$-lex-colimits which constitute the nature of $\Phi$-exactness.

In order to work efficiently with notions of $\Phi$-exactness, we shall make heavy use of \emph{embedding theorems}. Let us agree to call a $\V$-category $\C$ a \emph{$\V$-topos} if it is reflective in some $[\B^\op, \V]$ (with $\B$ small) via a finite-limit-preserving reflector. 
$\V$-toposes provide us with a basic source of $\Phi$-exact categories.
\begin{Prop}\label{prop:v-topos-phi-exact}
A $\V$-topos is $\Phi$-exact for any class of lex-weights $\Phi$.
\end{Prop}
\begin{proof}
See~\cite[Proposition 2.6]{Garner2011Lex-colimits}.
\end{proof}
The embedding theorem we shall make use of characterises general $\Phi$-exact categories in terms of their relation to $\V$-toposes. \begin{Thm}\label{thm:phi-exact-embedding}
For $\C$  
a small, finitely complete and $\Phi$-lex-cocomplete $\V$-category, the following are equivalent:
\begin{enumerate}[(a)]
\item $\C$ is $\Phi$-exact;
\item $\C$ admits a  finite-limit- and $\Phi$-lex-colimit-preserving full embedding into a $\Phi$-exact category;
\item $\C$ admits a  finite-limit- and $\Phi$-lex-colimit-preserving full embedding into a $\V$-topos.
\end{enumerate}
\end{Thm}
\begin{proof}
See~\cite[Theorem~4.1]{Garner2011Lex-colimits}.
\end{proof}
This embedding theorem will be used to, amongst other things, give elementary characterisations of particular $\Phi$-exactness notions. A prior, these characterisations will only be valid for small $\V$-categories; however, we may extend this validity to large $\V$-categories using the following result.
\begin{Prop}\label{prop:phi-small-reduction}
If $\Phi$ is a small class of lex-weights, then a finitely complete and $\Phi$-lex-cocomplete $\V$-category $\C$ is $\Phi$-exact if and only if each small, full, replete subcategory closed under finite limits and $\Phi$-lex-colimits is $\Phi$-exact.
\end{Prop}
\begin{proof}
See~\cite[Proposition~4.2]{Garner2011Lex-colimits}.
\end{proof}
The way we will use this result is as follows. Having found an elementary characterisation of $\Phi$-exactness that is valid for small $\C$, we observe that the nature of the characterisation in question is such that it will hold for a large $\C$ if and only if it does so for every small full, finite-limit- and $\Phi$-lex-colimit-closed subcategory. Applying the previous proposition, we conclude that the elementary characterisation holding for a small $\C$ remains valid for a large one.

\section{$\F$-regularity and $\F$-exactness}\label{sec:freg}
We now apply the theory of lex colimits to the study of regularity and exactness notions with respect to a kernel--quotient system. We begin by proving, as promised in the previous section:
\begin{Prop}\label{prop:ker-quot-lex-weights}
If $\F$ is a kernel--quotient system, then there are classes of lex-weights $\Phi_\F^\mathrm{reg}$ and $\Phi_\F^\mathrm{ex}$ such that a finitely complete category $\C$ is $\Phi_\F^\mathrm{reg}$-, respectively $\Phi_\F^\mathrm{ex}$-lex-cocomplete just when it admits $\F$-quotients of $\F$-kernels, respectively $\F$-congruences; and such that a finite-limit preserving functor between two such categories preserves $\Phi_\F^\mathrm{reg}$-, respectively $\Phi_\F^\mathrm{ex}$-lex-colimits just when it preserves $\F$-quotients of $\F$-kernels, respectively $\F$-congruences.
\end{Prop}
\begin{proof}
Considering first $\Phi_\F^\mathrm{reg}$, let $\I$ be the free $\V$-category with finite limits on an arrow $u \colon 1 \to 0$, let $W \colon \K \to \I$ be the $\F$-kernel of $u$, and let $\phi = \F(I\thg, 0) \in [\K^\op, \V]$ be the weight for $\F$-quotients; now take $\Phi_\F^\mathrm{reg} = \{\Lan_{W^\op} \phi \colon \I^\op \to \V\}$. For any finitely complete $\C$, to give a finite-limit preserving $D \colon \I \to \C$ is equally (to within isomorphism) to give the arrow $Du$ of $\C$. Since $D$ preserves finite limits, it preserves kernels, and so $DW \colon \K \to \C$ is the $\F$-kernel of $Du$. Thus $\Lan_{W^\op} \phi \star D \cong \phi \star DW$, if it exists, is the $\F$-quotient of the $\F$-kernel of $Du$.
%
%
It follows that a finitely complete $\V$-category is $\Phi_\F^\mathrm{reg}$-lex-cocomplete just when it admits $\F$-quotients of $\F$-kernels, and that a functor between two such categories will preserve $\Phi_\F^\mathrm{reg}$-lex-colimits just when it preserves these $\F$-quotients.

We now turn to $\Phi_\F^\mathrm{ex}$. We claim that there is a universal $\F$-congruence: a small, finitely complete $\V$-category $\mathbb C[\F]$
and a congruence $V \colon \K \to \mathbb C[\F]$, composition with which induces equivalences of categories
\begin{equation*}
\V\text-\cat{LEX}(\mathbb C[\F],\, \C) \xrightarrow{\simeq} \cat{Cong}(\C)
\end{equation*}
for every finitely complete $\V$-category $\C$. Given this, we may take $\Phi_\F^\mathrm{ex}$ to comprise the single weight $
\Lan_{V^\op} \phi \in [\mathbb C[\F]^\op, \V]$; now the same argument as before shows that a category has, or a functor preserves, $\Phi_\F^\mathrm{ex}$-lex-colimits just when it has, respectively preserves, $\F$-quotients of $\F$-congruences.

It remains to construct the universal $\F$-congruence. This may be done in many ways; we include one possible argument for the sake of completeness. Consider a small generating set $\{h_x \colon \phi_x \to \psi_x \mid x \in \Ss\}$ of $\F$-congruence axioms, and let $J \colon \K \to F(\K)$ exhibit $F(\K)$ as the free category with finite limits on $\K$. Viewing $\Ss$ as a discrete $\V$-category, we have a diagram in $\V$-$\cat{Cat}$ as on the left below, where the functors $A$ and $B$ send $x$ to $\{\psi_x, J\}$ and $\{\phi_x, J\}$ respectively, and $\Gamma$ has components $\Gamma_x = \{h_x, J\}$. Letting $F(\Ss)$ be the free category with finite limits on $\Ss$, we obtain from this a diagram in $\V\text-\cat{Lex}$ as on the right, where $\bar A$, $\bar B$ and $\bar \Gamma$ are the essentially-unique finite-limit-preserving extensions of $A$, $B$ and $\Gamma$ respectively.
\begin{equation*}
\cd[@C+2em]{
\Ss \ar@/^12pt/[r]^{A} \ar@/_12pt/[r]_{B} \dtwocell{r}{\Gamma} & F(\K)
} \qquad \qquad
\cd[@C+2em]{
F(\Ss) \ar@/^12pt/[r]^{\bar A} \ar@/_12pt/[r]_{\bar B} \dtwocell{r}{\bar \Gamma} & F(\K)
}
\end{equation*}
Now to give a congruence in the finitely complete $\C$ is equally to give $X \in [\K, \C]$ whose essentially-unique extension to a finite-limit-preserving $\bar X \colon F(\K) \to \C$ inverts the $2$-cell $\bar \Gamma$. By the results of~\cite[Section~5]{Blackwell1989Two-dimensional}, the $2$-category $\V\text-\cat{Lex}$ is bicocomplete,  and so 
there is a universal $L \colon F(\K) \to \mathbb C[\F]$ which inverts $\bar \Gamma$; because the inclusion $\V\text-\cat{Lex} \to \V\text-\cat{LEX}$ preserves small bicolimits, this $L$ is universal also with respect to \emph{large} $\V$-categories, and it follows that 
the desired universal $\F$-congruence is obtained as the composite $LJ \colon \K \to \mathbb C[\F]$.
\end{proof}

We are thus led to define a finitely complete $\V$-category to be \emph{$\F$-regular} if it is $\Phi_\F^\mathrm{reg}$-exact, and \emph{$\F$-exact} if it is $\Phi_\F^\mathrm{ex}$-exact. We observe that:
\begin{Prop}\label{prop:f-exact-f-regular}
An $\F$-exact category is $\F$-regular.
\end{Prop}
\begin{proof}
For any finitely complete $\C$, the $\V$-category $\Phi_\F^\mathrm{ex}(\C)$ is closed in $\P \C$ under finite limits and $\F$-quotients of $\F$-kernels, since by Proposition~\ref{prop:f-ker-is-f-cong}, every $\F$-kernel is an $\F$-congruence. Thus $\Phi_\F^\mathrm{reg}(\C) \subseteq \Phi_\F^\mathrm{ex}(\C)$ and the (full) inclusion $J$ preserves finite limits. Now if $\C$ is $\F$-exact, then the embedding $\C \to \Phi_\F^\mathrm{ex}(\C)$ admits a finite-limit-preserving left adjoint $L$, whence $\C \to \Phi_\F^\mathrm{reg}(\C)$ admits the finite-limit-preserving left adjoint $LJ$ and so $\C$ is $\F$-regular. 
\end{proof}


The following result shows that an $\F$-regular or $\F$-exact $\V$-category inherits many good properties that the $\F$-kernel--quotient system may have in $\V$.

\begin{Prop}\label{prop:ker-quot-immediate-fregular}
\begin{enumerate}[(a)]
\item If kernel--quotient factorisations for $\F$ converge immediately in $\V$, then they do so in every $\F$-regular $\V$-category.
\item If kernel--quotient factorisations are stable under pullback in $\V$, then they are also stable in any $\F$-regular $\V$-category; it follows that effective $\F$-quotient maps are stable under pullback.
\item If $\F$-congruences are effective in $\V$, then they are so in every $\F$-exact $\V$-category.
\end{enumerate}
\end{Prop}
\begin{proof}
For (a), if kernel--quotient factorisations converge immediately in $\V$, they also do so in any $\V$-category of the form $\P \C$, since finite limits and colimits there are computed pointwise. It follows that they converge immediately in any $\V$-category of the form $\Phi_\F^\mathrm{reg}(\C)$, since such a category is closed under the formation of $\F$-kernels and of $\F$-quotients of $\F$-kernels in $\P \C$. Finally, if $\C$ is $\F$-regular, so that the embedding $Z \colon \C \to \Phi_\F^\mathrm{reg}(\C)$ admits a left exact left adjoint $L$, then the kernel--quotient factorisation of $f$ in $\C$ may be computed by first forming the corresponding factorisation of $Zf$ in $\Phi_\F^\mathrm{reg}(\C)$ and then applying the reflector $L$. Since $L$ preserves finite limits, it preserves $\F$-monics, and so kernel--quotient factorisations in $\C$ converge immediately, since they do so in $\Phi_\F^\mathrm{reg}(\C)$.

For (b), a similar argument to the one just given shows that if kernel--quotient factorisations are stable under pullback in $\V$, then they are also stable in every $\P \C$, thus in every $\Phi_\F^\mathrm{reg}(\C)$, and thus in every $\F$-regular $\V$-category $\C$. Now if $f \colon A \to B$ is an effective quotient map in such a $\C$, then on forming its kernel--quotient factorisation $f = me$, the second component $m$ is invertible; pulling back along some $g \colon B' \to B$, we obtain a kernel--quotient factorisation $f' = m' e'$ of the pullback $f'$ of $f$ with second component invertible; whence $f'$, like $f$, is an effective $\F$-quotient.

For (c), if every $\F$-congruence is the $\F$-kernel of its $\F$-quotient in $\V$, then the same is true in every $\V$-category $\P \C$, since kernel-data $X \in [\K, \P \C]$ is an $\F$-congruence if and only if it is pointwise so. Arguing as before, it follows that the same is true in every $\Phi_\F^\mathrm{ex}(\C)$, and thus in every $\F$-exact $\C$, as required.
\end{proof}

\begin{Rk}
We saw in Remark~\ref{rk:not-immediate} that $\F$-kernel--quotient factorisations need not converge immediately in a general $\V$-category; the preceding result shows that they will do so in an $\F$-regular category---as in the motivating kernel pair--coequaliser case---\emph{so long as they converge immediately in $\V$ itself}. This is not automatic. For example, take $\V = \cat{Cat}$ and $\F$ the $2$-category generated by the graph
\begin{equation*}
\cd[@C+1em]{
 2 \ar@/^12pt/[r]^-{u} \ar@/_12pt/[r]_-{v} \dtwocell{r}{\alpha} & 1 \ar[r]^w & 0
}
\end{equation*}
subject to the condition that $w\alpha$ be invertible. In this case, the $\F$-monics in $\cat{Cat}$ are the conservative functors, whilst the $\F$-kernel--quotient factorisation of a functor $F \colon \C \to \D$ has as interposing object the localisation $\C[\Sigma^{-1}]$, where $\Sigma$ is the class of maps in $\C$ which are inverted by $F$. Now the example given in Remark~3.3 of~\cite{Carboni1994Modulated} provides an $F$ for which the second half  of this factorisation is not conservative. If we instead take $\V = \cat{Cat}_\mathrm{pb}$, the cartesian closed category of categories with pullbacks, then kernel--quotient factorisations for the $\cat{Cat}_\mathrm{pb}$-enriched version of the above $\F$ \emph{do} converge immediately; the paper~\cite{Benabou1989Some} develops these ideas further.
\end{Rk}

An important class of regular and Barr-exact categories are obtained from universal algebra. By an \emph{algebraic theory}, we mean a small category $\A$ with finite products, whilst a \emph{model} of $\A$ in a category $\C$ with finite products is a finite-product-preserving functor $\A \to \C$. It is easy to show that, if $\C$ is a regular or Barr-exact category, then so is the category $\cat{FP}(\A, \C)$ of $\A$-models in $\C$. The final result of this section generalises this fact.

\begin{Prop}\label{prop:sifted}
Let $\F$ be a kernel--quotient system whose $\F$-quotient morphisms are closed under finite products in $\V$, and let
$\A$ be a small $\V$-category with finite products.
\begin{enumerate}[(a)]
\item Let $\F$-quotients be effective in $\V$. If $\C$ is $\F$-regular, then so is $\cat{FP}(\A, \C)$.
\item Let $\F$-congruences be effective in $\V$. If $\C$ is $\F$-exact, then so is $\cat{FP}(\A, \C)$.
\end{enumerate}
\end{Prop}
The hypothesis on $\F$-quotient maps is most easily verified when kernel--quotient factorisations converge immediately in $\V$, so that $\F$-quotient maps in $\V$ coincide with $\F$-strong epis and are thus closed under composition and identities. Then the terminal object $\id_1 \colon 1 \to 1$ of $[\mathbf 2, \V]$ is always an $\F$-quotient map; and a familiar argument shows that $\F$-quotient maps are closed under binary products if they are stable under pullback. For if $f \colon A \to B$ and $g \colon C \to D$ are $\F$-quotients, then so are $f \times C$ and $B \times g$ (the pullbacks of $f$ and $g$ along $C \to 1$ and $B \to 1$), and hence also their composite $f \times g = (B \times g) \circ (f \times C)$.
\begin{proof}
We first prove (a). If $\C$ is $\F$-regular, then clearly so too is $[\A, \C]$; now by Theorem~\ref{thm:phi-exact-embedding}, it suffices to show that $\cat{FP}(\A, \C)$ is closed in $[\A, \C]$ under finite limits and $\F$-quotients of $\F$-kernels. Closure under finite limits is clear, since these commute with finite products in $\C$; closure under $\F$-quotients of $\F$-kernels will follow similarly if we can show that these also commute with finite products in $\C$, or equivalently, that each $n$-ary product functor $\Pi \colon \C^n \to \C$ preserves $\F$-quotients of $\F$-kernels.
By the argument of the previous proposition, $\F$-quotient morphisms are closed under finite products in the $\F$-regular $\C$, since they are so in $\V$. Thus
each  $\Pi \colon \C^n \to \C$ preserves $\F$-quotient maps as well as finite limits. Since $\C^n$ and $\C$ are $\F$-regular, they admit effective $\F$-quotients of $\F$-kernels, since $\V$ does; whence by Proposition~\ref{prop:preserve-quotients-preserve-colims}(a), each $\Pi \colon \C^n \to \C$ preserves $\F$-quotients of $\F$-kernels. 
The argument for (b) is identical, but now using the second part of Proposition~\ref{prop:preserve-quotients-preserve-colims}.
\end{proof}

\section{$2$-dimensional kernel--quotient systems}\label{sec:2dkerquot}
We now commence on the second objective of this paper: the study of particular notions of two-dimensional regularity and exactness. In this section, we will describe three $\cat{Cat}$-enriched kernel--quotient systems, whose induced factorisations in $\cat{Cat}$ correspond to the three factorisation systems described in the introduction; in the following section, we will identify the corresponding notions of $\F$-regularity and $\F$-exactness. 

\subsection{(Bijective on objects, fully faithful)}
Let $\F_\mathrm{bo}$ be the $2$-category generated by the graph
\begin{equation*}
\cd{
3 \ar@<6pt>[r]^{p} \ar[r]|{m} \ar@<-6pt>[r]_{q} &
2 \ar@<6pt>[r]^{d} \ar@{<-}[r]|{i} \ar@<-6pt>[r]_{c} &
1 \ar[r]^-w & 0
}
\end{equation*}
together with a $2$-cell $\theta \colon wd \Rightarrow wc$, subject to the simplicial identities $di = ci = 1$, $cp = dq$, $dm = dp$ and $cm = cq$, and the cocycle conditions $\theta i = 1_w$ and $(\theta q)(\theta p) = \theta m$. From this we obtain the following kernel--quotient system. The $\F_\mathrm{bo}$-kernel of a map 
$f \colon A \to B$ in a finitely complete $\C$ is given by:
\begin{equation*}
\cd{
f \mathord{\mid} f \mathord{\mid} f \ar@<6pt>[r]^-{p} \ar[r]|-{m} \ar@<-6pt>[r]_-{q} &
f \mathord{\mid} f \ar@<6pt>[r]^-{d} \ar@{<-}[r]|-{i} \ar@<-6pt>[r]_-{c} &
A\rlap{ .}
}
\end{equation*}
In the case $\C = \cat{Cat}$, the category $f \mathord{\mid} f$ has objects $(x, y \in A, \alpha \colon fx \to fy \in B)$, whilst $f \mathord{\mid} f \mathord{\mid} f$ has objects $(x,y,z \in A,\alpha \colon fx \to fy, \beta \colon fy \to fz \in B)$. The functors $d$ and $c$ send $(x,y,\alpha)$ to $x$ and $y$ respectively; $i$ sends $x$ to $(x,x,1_{fx})$; and $p$, $m$ and $q$ send $(x,y,z,\alpha,\beta)$ to $(x,y,\alpha)$, to $(x,z,\beta\alpha)$, and to $(y,z,\beta)$ respectively. The definition in a general $2$-category follows representably.

The $\F_\mathrm{bo}$-quotient of kernel-data $X \in [\K_\mathrm{bo}, \C]$ is the universal \emph{codescent cocone} under $X$; such a cocone comprises an
object $Q \in \C$, a morphism $q \colon X1 \to Q$, and a $2$-cell $\theta \colon q.Xd \Rightarrow q.Xc$ satisfying the two cocycle conditions $\theta.Xi = 1$ and $(\theta.Xp)(\theta.Xq) = \theta.Xm$. In a sufficiently cocomplete $2$-category, we may construct the  $\F_\mathrm{bo}$-quotient of $X$ by first forming the coinserter of $Xd$ and $Xc$, and then taking two coequifiers imposing the cocycle conditions; note, however, that the quotient may still exist even if these intermediate colimits do not.

A morphism $f$ is an $\F_\mathrm{bo}$-monic just when the comparisons $\gamma_f \colon A^\mathbf 2 \to f \mathord{\mid} f$ and $\xi_f \colon A^\mathbf 3 \to f \mathord{\mid} f \mathord{\mid} f$, given in $\cat{Cat}$ by $\gamma_f(x,y,\alpha) = (x,y,f\alpha)$ and $\xi_f(x,y,z,\alpha,\beta) = (x,y,z,f\alpha,f\beta)$, are invertible; in fact, the invertibility of $\gamma_f$ implies that of $\xi_f$. We call such morphisms \emph{fully faithful}; when $\C = \cat{Cat}$ they are precisely the fully faithful functors, whilst in a general $2$-category $\C$, they are by Proposition~\ref{prop:f-map-props}(a) the morphisms $f$ such that $\C(X,f)$ is a fully faithful functor for every $X \in \C$.

\begin{Prop}\label{prop:ker-quot-bo-cat}
Kernel--quotient factorisations for $\F_\mathrm{bo}$ converge immediately in $\cat{Cat}$, where they yield the (bijective on objects, fully faithful) factorisation of a functor; in particular, they are stable under pullback.
\end{Prop}
\begin{proof} 
See, for example,~\cite[Proposition~3.1]{Street2004Categorical}; alternatively, this follows from the explicit description of $\F_\mathrm{bo}$-quotients of $\F_\mathrm{bo}$-congruences (and so in particular, $\F_\mathrm{bo}$-kernels) in $\cat{Cat}$ given in Proposition~\ref{prop:cateads} below.
\end{proof}
\begin{Cor}\label{cor:ker-quot-bo-2topos}
Kernel--quotient factorisations for $\F_\mathrm{bo}$ are stable and converge immediately in any $\F_\mathrm{bo}$-regular $2$-category (in particular, in any $2$-topos); thus, the classes of $\F_\mathrm{bo}$-strong epis, $\F_\mathrm{bo}$-quotients and effective $\F_\mathrm{bo}$-quotients coincide and are pullback-stable in any $\F_\mathrm{bo}$-regular $2$-category.
\end{Cor}
\begin{proof}
By Propositions~\ref{prop:v-topos-phi-exact}, \ref{prop:ker-quot-immediate-fregular},  and \ref{prop:ker-quot-bo-cat}.
\end{proof}
However, in a general $2$-category, we have that:
 \begin{Prop}\label{prop:codescent-not-effective}
$\F_\mathrm{bo}$-quotient maps need not be effective, even in a locally finitely presentable (and hence complete and cocomplete) $2$-category.
\end{Prop}
\begin{proof}
The category $\cat{Ab}$ of abelian groups is locally finitely presentable; whence, by~\cite[Theorem~4.5]{Kelly2001V-Cat}, so too is the $2$-category $\cat{Ab}\text-\cat{Cat}$. We will show that not every $\F_\mathrm{bo}$-kernel is effective in $\cat{Ab}\text-\cat{Cat}$; the result then follows from Proposition~\ref{prop:effectivity-ker-quot}. 

Let $\phi \colon \R \to \Ss$ be an $\cat{Ab}$-functor between one-object $\cat{Ab}$-categories. Writing $R$ and $S$ for the rings $\R(\star, \star)$ and $\Ss(\star, \star)$, we will compute the $\F_\mathrm{bo}$-kernel of $\phi$, and the $\F_\mathrm{bo}$-quotient of that.
The comma object $\phi | \phi$ is the $\cat{Ab}$-category whose objects are elements $s \in S$, and whose hom-objects are given by
\[
(\phi | \phi)(s_1, s_2) = \{(r_1, r_2) \in R \oplus R \,\mid\,\phi(r_2) s_1 = s_2 \phi(r_1)\}\rlap{ ;}
\]
the triple comma object $\phi | \phi | \phi$ has as objects, pairs $(s_1, s_2) \in S \times S$, and hom-objects given similarly to above; the maps $p, m, q \colon \phi|\phi|\phi \to \phi|\phi$ are given by first projection, multiplication and second projection, and the map $i \colon \R \to \phi | \phi$ picks out the multiplicative unit of $S$.
We now describe what it is to give a codescent cocone with vertex $\C$ under the $\F_\mathrm{bo}$-kernel of $\phi$. Firstly, we must give an $\cat{Ab}$-functor $\R \to \C$: which is equally to give an object $x \in \C$ and a ring homomorphism $\gamma \colon R \to \C(x,x)$. Next, we must give an $\cat{Ab}$-natural transformation 
\[
\cd[@R-1.8em]{
& \R \ar[dr] \\
f | f \ar[ur]^d \ar[dr]_c \dtwocell{rr}{\theta} & & \C \\ &
\R \ar[ur]}\rlap{ .}
\]
The components of this are elements $(\theta(s) \in \C(x,x) \mid s \in S)$, and naturality says that if $r_1,r_2 \in R$ and $s_1, s_2 \in S$ satisfy
$\phi(r_2) s_1 = s_2 \phi(r_1)$, then $\gamma(r_2) \theta(s_1) = \theta(s_2) \gamma(r_1)$. Finally, the two cocycle conditions $\theta i = 1$ and $\theta q . \theta p = \theta m$ say that $\theta(1) = 1$ and that $\theta(s_1s_2) = \theta(s_1)\theta(s_2)$ for all $s_1,s_2 \in S$. Note that, in the presence of these last two conditions, the naturality of $\theta$ is equivalent to the condition that $\theta(\phi(r)) = \gamma(r)$ for all $r \in R$, and this condition in turn implies that $\theta(1) = 1$. From this calculation, it follows that the $\F_\mathrm{bo}$-quotient of the $\F_\mathrm{bo}$-kernel of $\phi$ is the canonical $\cat{Ab}$-functor $\psi \colon \R \to \Ss'$ where $\Ss'$ is the one-object $\cat{Ab}$-category with $\Ss'(\star, \star)$ the ring 
\[S' = R[x_s \,|\, s \in S] / \spn {x_s x_t - x_{st} \mid s, t \in S} \cup \spn{x_{\phi(r)} - r \mid r \in R}\ \rlap.\]

Now, if the $\F_\mathrm{bo}$-kernel of $\phi$ is to be effective, then it must coincide with the $\F_\mathrm{bo}$-kernel of $\psi$;  and since $\psi | \psi$ has as objects elements of $S'$, this cannot happen unless $S \cong S'$ as sets. Consider now $\phi \colon \R \to \Ss$ obtained from the evident homomorphism between $R = \mathbb F_2 = \mathbb Z / \spn 2$ and $S = \mathbb F_4 = R[x] / \spn{x^2+x+1}$. In this case, we calculate that $S' = R[x] / \spn{x^3+1}$;  thus, since $S$ has four elements and $S'$ eight, the $\F_\mathrm{bo}$-kernel of this $\phi$ is not effective, as desired.
\end{proof}

The following result identifies the $\F_\mathrm{bo}$-congruences explicitly, showing that they are  the \emph{cateads} of~\cite{Bourn20102-categories,Bourke2010Codescent}.
\begin{Prop}\label{prop:cateads}
A diagram $X \in [\K_\mathrm{bo}, \C]$ is an $\F_\mathrm{bo}$-congruence if and only if:
\begin{enumerate}[(a)]
\item $X$ is the truncated nerve of an internal category in $\C$;
\item The span $Xd \colon X1 \leftarrow X2 \rightarrow X1 \colon Xc$ is a two-sided discrete fibration.
\end{enumerate}
$\F_\mathrm{bo}$-congruences are effective in $\cat{Cat}$, and hence in every $\F_\mathrm{bo}$-exact $2$-category.
\end{Prop}
In the statement of this result, we recall that a span of functors $p \colon \C \leftarrow \E \rightarrow \D \colon q$ is called a \emph{two-sided discrete fibration} if:
\begin{itemize}
\item for every $e \in \E$ and $\alpha \colon c \to pe$ in $\C$, there exists a unique $\bar \alpha \colon \alpha^\ast e \to e$ in $\E$ with $p(\bar \alpha) = \alpha$ and $q(\bar \alpha) = 1_{qe}$, and moreover $\bar \alpha$ is cartesian for $p$;
\item for every $e \in \E$ and $\beta \colon qe \to d$ in $\D$, there exists a unique $\bar \beta \colon e \to \beta_\ast e$ in $\E$ with $p(\bar \beta) = 1_{pe}$ and $q(\bar \beta) = \beta$, and moreover $\bar \beta$ is opcartesian for $q$.
\end{itemize}
A span in the general $2$-category $\A$ is a two-sided discrete fibration if it is sent to one in $\cat{Cat}$ by each representable $\A(A, \thg)$.
\begin{proof}[Proof of Proposition~\ref{prop:cateads}]
It is a straightforward exercise to construct a set $\Ss$ of morphisms between finitely 
%
%
presentable objects of $[\K, \cat{Cat}]$ such that an object $X \in [\K, \C]$
inverts $\{h, X\}$ for each $h \in \Ss$ just when it satisfies (a) and (b). 
Since every $\F_\mathrm{bo}$-kernel in $\cat{Cat}$ is known to satisfy (a) and (b), the set $\Ss$ just described is in fact a set of $\F_\mathrm{bo}$-congruence axioms; and the desired characterisation of the $\F_\mathrm{bo}$-congruences will follow from Proposition~\ref{prop:f-congruence-generating} if we can show that $\Ss$ is generating. This will in turn follow from Proposition~\ref{prop:f-effective-cong} if we can show that every $X \in [\K, \cat{Cat}]$ satisfying (a) and (b) is in fact the $\F_\mathrm{bo}$-kernel of its own $\F_\mathrm{bo}$-quotient. 

This is proven, for example, in~\cite[Proposition~2.83]{Bourke2010Codescent}; we recall the outline of the proof. Given $X \in [\K, \cat{Cat}]$ satisfying (a) and (b), its $\F_\mathrm{bo}$-quotient is $q \colon X1 \to Q$, where $Q$ is the category obtained by applying the pullback-preserving functor $\mathrm{ob} \colon \cat{Cat} \to \cat{Set}$ to the internal category $X$, and $q$ is the identity on objects, and takes a morphism $\alpha \in X1(x,y)$ to the morphism of $Q$ (i.e., object of $X2$) obtained as the codomain of the unique opcartesian lifting of the map $\alpha \colon x \to y$ at the object $(Xr)(x) \in X2$. The codescent $2$-cell $\theta \colon q.Xd \Rightarrow q.Xc \colon X2 \to Q$ has component at $\gamma \in X2$ given by $\gamma$ itself, seen as a map $(Xd)(\gamma) \to (Xc)(\gamma)$ of $Q$. It is now not hard to see that $\theta$ in fact exhibits $X2$ as the comma object $q \mathord{\mid} q$, from which it follows that $X$ is the $\F_\mathrm{bo}$-kernel of its own $\F_\mathrm{bo}$-quotient, as required.
\end{proof}

\subsection{(Surjective on objects, injective on objects and fully faithful)}
Let $\F_{\mathrm{so}}$ be the $2$-category obtained from $\F_{\mathrm{bo}}$ by adjoining a new object $2'$, a new morphism $j \colon 2' \to 2$
and new equations $wdj = wcj$ and $\theta j = 1_{wdj}$.
%
From this we obtain the following kernel--quotient system. The $\F_\mathrm{so}$-kernel of a map 
$f \colon A \to B$ in a finitely complete $\C$ is given by:
%
%
%
\begin{equation*}
\cd{
& A \times_B A \ar[d]_j 
\\
f \mathord{\mid} f \mathord{\mid} f \ar@<6pt>[r]^-{p} \ar[r]|-{m} \ar@<-6pt>[r]_-{q} &
f \mathord{\mid} f \ar@<6pt>[r]^-{d} \ar@{<-}[r]|-{i} \ar@<-6pt>[r]_-{c} &
A}
\end{equation*}
wherein the bottom row is the $\F_\mathrm{bo}$-kernel, $A \times_B A$ is the pullback of $f$ along itself, and $j$ is the morphism which in $\cat{Cat}$ sends $(x,y)$ with $fx = fy$ to $(x,y,1 \colon fx \to fy)$.

The $\F_\mathrm{so}$-quotient of kernel-data $X \in [\K_\mathrm{so}, \C]$ is a codescent cocone $(Q, q, \theta)$ under the underlying $\F_\mathrm{bo}$-kernel-data which is universal amongst cocones for which  $\theta.Xj$ is an identity $2$-cell. In a sufficiently cocomplete $2$-category, we may construct the $\F_\mathrm{so}$-quotient by first forming the $\F_\mathrm{bo}$-quotient of the underlying $\F_\mathrm{bo}$-kernel-data, and then taking a coidentifier enforcing the  compatibility with $Xj$.

A map $f$ is an $\F_\mathrm{so}$-monic just when the maps $\delta_f \colon A \to A \times_B A$ and $\gamma_f \colon 
A^\mathbf 2 \to f \mathord{\mid} f$, defined in $\cat{Cat}$ by $\delta_f(a) = (a,a)$ and $\gamma_f(\alpha \colon x \to y) = (x,y,f\alpha \colon fx \to fy)$, are invertible. We call such morphisms \emph{full monics}, since in $\cat{Cat}$, they are precisely the injective on objects and fully faithful functors; as before, their definition in a general $2$-category is representable. 
We will call $\F_\mathrm{so}$-strong epis \emph{acute} following~\cite{Street1982Two-dimensional}; the following cancellativity property of acute maps, enhancing Proposition~\ref{prop:f-map-props}(d), will come in useful in what follows.
\begin{Prop}\label{prop:strongcanc}
In a $2$-category $\C$, if $gf$ is acute then so too is $g$.
\end{Prop}
\begin{proof}
This is an easy consequence of the fact that full monics are monic.
\end{proof}
 
\begin{Prop}\label{prop:ker-quot-so-cat}
Kernel--quotient factorisations for $\F_\mathrm{so}$ converge immediately in $\cat{Cat}$, where they are given by the (surjective on objects, injective on objects and fully faithful) factorisation of a functor; in particular, they are stable under pullback.
\end{Prop}
\begin{proof}
This follows immediately from the explicit descriptions of $\F_\mathrm{so}$-quotients of $\F_\mathrm{so}$-congruences in $\cat{Cat}$ given in the proof of Proposition~\ref{prop:mixed-cateads} below.
\end{proof}

 \begin{Cor}\label{cor:ker-quot-so-2topos}
Kernel--quotient factorisations for $\F_\mathrm{so}$ are stable and converge immediately in any $\F_\mathrm{so}$-regular $2$-category (in particular, in any $2$-topos); thus, the classes of $\F_\mathrm{so}$-strong epis, $\F_\mathrm{so}$-quotients and effective $\F_\mathrm{so}$-quotients coincide and are pullback-stable in any $\F_\mathrm{so}$-regular $2$-category.
\end{Cor}
\begin{proof}
By Propositions~\ref{prop:v-topos-phi-exact}, \ref{prop:ker-quot-immediate-fregular} and \ref{prop:ker-quot-so-cat}.
\end{proof}
But once again, in a general $2$-category, we have that:
 \begin{Prop}\label{prop:mixed-codescent-not-effective}
$\F_\mathrm{so}$-quotient maps need not be effective, even in a locally finitely presentable (and hence complete and cocomplete) $2$-category.
\end{Prop}
(Note that this corrects an error in~\cite[\S1.14]{Street1982Two-dimensional}, wherein it is claimed that $\F_\mathrm{so}$-quotient maps are always effective.)
\begin{proof}
As in the proof of Proposition~\ref{prop:codescent-not-effective}, we consider the $2$-category $\cat{Ab}\text-\cat{Cat}$. Let $\phi \colon \R \to \Ss$ be an $\cat{Ab}$-functor between one-object $\cat{Ab}$-categories; we claim that the $\F_\mathrm{so}$-quotient of its $\F_\mathrm{so}$-kernel coincides with the $\F_\mathrm{bo}$-quotient of its $\F_\mathrm{bo}$-kernel. Thus if the $\F_\mathrm{bo}$-kernel of $\phi$ is not effective, then neither is its $\F_\mathrm{so}$-kernel; and so the result follows by taking the counterexample from Proposition~\ref{prop:codescent-not-effective}.
To prove the claim, observe that to give a cocone with vertex $\C$ under the $\F_\mathrm{so}$-kernel of $\phi$ is to give a cocone under the $\F_\mathrm{bo}$-kernel of $\phi$ for which the composite
\[
\cd[@R-1.7em]{
& & \R \ar[dr] \\
\R \times_{\Ss} \R \ar[r] & \phi | \phi \ar[ur]^d \ar[dr]_c \dtwocell{rr}{\theta} & & \C \\ &
&\R \ar[ur]}
\]
is an identity $2$-cell. Now $\R \times_{\Ss} \R$ has a single object, which is sent to the multiplicative unit of $S$ in $\phi | \phi$; and so this condition states that $\theta(1) = 1$ in $\C(x,x)$. Since this condition is already verified by a cocone under the $\F_\mathrm{bo}$-kernel, the $\F_\mathrm{so}$-quotient and the $\F_\mathrm{bo}$-quotient coincide as claimed.
\end{proof}

The following result gives an elementary characterisation of the $\F_\mathrm{so}$-congruences; the conditions it isolates were first stated in~\cite[\S1.8]{Street1982Two-dimensional}.
\begin{Prop}\label{prop:mixed-cateads}
A diagram $X \in [\K_\mathrm{so}, \C]$ is an $\F_\mathrm{so}$-congruence if and only if:
\begin{enumerate}[(a)]
\item Its restriction to an object of $[\K_\mathrm{bo}, \C]$ is an $\F_\mathrm{bo}$-congruence;
\item $X(dj), X(cj) \colon X2' \rightrightarrows X1$ is an equivalence relation in $\C$;
\item $Xj$ is full monic;
\item The graph morphism
\begin{equation*}
\cd[@C+2em]{
X2' \ar[r]^{Xj} \ar@<-3pt>[d]_{X(dj)} \ar@<3pt>[d]^{X(cj)} &
X2 \ar@<-3pt>[d]_{Xd} \ar@<3pt>[d]^{Xc} &
\\
X1 \ar[r]_{1} & X1
}
\end{equation*}
is an internal functor.
\end{enumerate}
$\F_\mathrm{so}$-congruences are effective in $\cat{Cat}$, and hence in every $\F_\mathrm{so}$-exact $2$-category.
\end{Prop}
\begin{proof}
As before, we may easily construct
 a set $\Ss$ of morphisms between finitely presentable objects of $[\K_\mathrm{so}, \cat{Cat}]$ such that an object $X \in [\K_\mathrm{so}, \C]$
inverts $\{h, X\}$ for each $h \in \Ss$ just when it satisfies (a)--(d).
Every $\F_\mathrm{so}$-kernel in $\cat{Cat}$ satisfies (a)--(d), so that $\Ss$ is in fact a set of $\F_\mathrm{so}$-congruence axioms; and the desired characterisation of the $\F_\mathrm{so}$-congruences will follow from Propositions~\ref{prop:f-congruence-generating} and~\ref{prop:f-effective-cong} if we can show that every $X \in [\K_\mathrm{so}, \cat{Cat}]$ satisfying (a)--(d) is in fact the $\F_\mathrm{so}$-kernel of its own $\F_\mathrm{so}$-quotient.

This is proven as~\cite[Theorem~2.3]{Street1982Two-dimensional};  let us once more recall the outline of the proof. Given $X \in [\K_\mathrm{so}, \cat{Cat}]$ satisfying (a)--(d), its $\F_\mathrm{so}$-quotient in $\cat{Cat}$ is constructed as follows. By (a), $X$ has an underlying $\F_\mathrm{bo}$-congruence; we start by forming the $\F_\mathrm{bo}$-quotient $q \colon X1 \to Q$ of this. Now by (b), we obtain from $X2' \rightrightarrows X1$ an equivalence relation $\sim$ on the set of objects of $Q$; and by (c) and (d), we have an identity-on-objects functor from the category $E$ instantiating this equivalence relation into $Q$. Thus, whenever $x \sim y \in Q$, there is a specified morphism $\phi_{xy} \colon x \to y$ of $Q$, such that $\phi_{xx} = 1_x$ and $\phi_{yz} \circ \phi_{xy} = \phi_{xz}$. We now define the $\F_\mathrm{so}$-quotient of $X$ to be the composite of $q \colon X_1 \to Q$ with the functor $k \colon Q \to R$ obtained by quotienting the object set of $Q$ by $\sim$,
and quotienting the morphism set by the equivalence relation for which $\alpha \colon x \to y$ and $\beta \colon w \to z$ are related just when $x \sim w$ and $y \sim z$ and the square
\begin{equation*}
\cd{
x \ar[r]^{\phi_{xw}} \ar[d]_\alpha &
w \ar[d]^\beta \\
y \ar[r]_{\phi_{yz}} & z
}
\end{equation*}
commutes. Direct calculation shows that $X$ is the $\F_\mathrm{so}$-kernel of $kq$, as required.
\end{proof}

\subsection{(Bijective on objects and full, faithful)}
Consider the $2$-category $\F_\mathrm{bof}$ generated by the left-hand graph in:
\begin{equation*}
\cd[@C+1em]{
 2 \ar@/^12pt/[r]^-{u} \ar@/_12pt/[r]_-{v} \dtwocell[0.32]{r}{\alpha} \dtwocell[0.6]{r}{\beta} & 1 \ar[r]^w & 0
} \qquad \quad
\cd[@C+1em]{
 \mathrm{Eq}(f) \ar@/^12pt/[r]^{u} \ar@/_12pt/[r]_{v} \dtwocell[0.32]{r}{\alpha} \dtwocell[0.6]{r}{\beta} & A \ar[r]^f & B\rlap{ ,}
}
\end{equation*}
subject to the relation $w\alpha = w\beta$. 
From this we obtain the following kernel--quotient system. The $\F_\mathrm{bof}$-kernel of a map 
$f \colon A \to B$ in a finitely complete $\C$ is given as on the right above;
in $\cat{Cat}$, $\mathrm{Eq}(f)$ is the category with objects, parallel pairs of morphisms $(a, b \colon x \rightrightarrows y)$ in $A$ such that $f(a) = f(b)$ in $B$; at such an object, the functors $u$ and $v$ take values $x$ and $y$ respectively, whilst $\alpha$ and $\beta$ have respective components $a$ and $b$.
%
The $\F_\mathrm{bof}$-quotient of kernel-data $X \in [\K_\mathrm{bof}, \C]$ is its \emph{coequifier}: the universal $1$-cell $q \colon X1 \to Q$ with $q.X\alpha = q.X\beta$.

A morphism $f \colon A \to B$ of $\C$ is an $\F_\mathrm{bof}$-monic just when the canonical comparison map $A^\mathbf 2 \to \mathrm{Eq}(f)$ is invertible. We call such morphisms \emph{faithful}; when $\C = \cat{Cat}$, they are precisely the faithful functors, whilst in the general $\C$, they are the morphisms 
 $f$ for which $\C(X, f)$ is faithful for all $X \in \C$. 
 
\begin{Prop}\label{prop:ker-quot-bof-cat}
Kernel--quotient factorisations for $\F_\mathrm{bof}$ converge immediately in $\cat{Cat}$, where they are given by the (bijective on objects and full, faithful) factorisation of a functor; in particular, they are stable under pullback.
\end{Prop}
\begin{proof}
This follows immediately from the explicit descriptions of $\F_\mathrm{bof}$-kernels and of $\F_\mathrm{bof}$-quotients of $\F_\mathrm{bof}$-congruences in $\cat{Cat}$ given in the proof of Proposition~\ref{prop:congruences} below.
\end{proof}

\begin{Cor}\label{cor:ker-quot-bof-2topos}
Kernel--quotient factorisations for $\F_\mathrm{bof}$ are stable and converge immediately in any $\F_\mathrm{bof}$-regular $2$-category (in particular, in any $2$-topos); thus, the classes of $\F_\mathrm{bof}$-strong epis and $\F_\mathrm{bof}$-quotient maps coincide and are pullback-stable in any $\F_\mathrm{bof}$-regular $2$-category.
\end{Cor}
\begin{proof}
By Propositions~\ref{prop:v-topos-phi-exact}, \ref{prop:ker-quot-immediate-fregular} and \ref{prop:ker-quot-bof-cat}.
\end{proof}
In the preceding result, we have not made mention of the \emph{effective} $\F_\mathrm{bof}$-quotients. This is because, by contrast to the preceding examples, we have:
\begin{Prop}
All $\F_\mathrm{bof}$-quotient maps are effective.
\end{Prop}
\begin{proof}
The equikernel of $f \colon A \to B$ can be constructed as a subobject of $A^{\mathbb P}$ (where $\mathbb P$ is the parallel pair category $\bullet \rightrightarrows \bullet$). Consequently, 
given a commutative triangle of morphisms as on the left in
\begin{equation*}
\cd{
 & A \ar[dl]_f \ar[dr]^g \\
 B \ar[rr]_-h & & C
} \qquad \qquad \quad
\cd[@!C@C-1.8em@R-0.5em]{
\mathrm{Eq}(f) \ar[rr]^-{\mathrm{Eq}(1, h)} \ar[dr] & &
\mathrm{Eq}(g) \ar[dl] \\ &
A^{\mathbb P}
}
\end{equation*}
we have on taking equikernels a commutative diagram as on the right. Both diagonal arrows are monomorphisms, whence also the top arrow; so in particular, taking $f = QKg$ and $h = \epsilon_g$, we conclude that $K\epsilon \colon KQK \Rightarrow K$ is monomorphic. But it is also split epimorphic, with section $\eta K$, and so both $K \epsilon$ and $\eta K$ are invertible. Thus all $\F_\mathrm{bof}$-kernels are effective, whence, by Proposition~\ref{prop:effectivity-ker-quot}, so too are all $\F_\mathrm{bof}$-quotient maps.
\end{proof}

We now identify the $\F_\mathrm{bof}$-congruences.
\begin{Prop}\label{prop:congruences}
A diagram $X \in [\K_\mathrm{bof}, \C]$ is an $\F_\mathrm{bof}$-congruence if and only if:
\begin{enumerate}[(a)]
\item The induced map  $(X\alpha, X\beta) \colon X2 \rightarrow X1^{\mathbb P}$ (with $\mathbb P$ the parallel pair category $\bullet \rightrightarrows \bullet$) is full monic;\vskip0.5\baselineskip
\item The induced pair
\begin{equation}\label{eq:equiv-relation-diagram}
\cd[@C-1em]{
X2 \ar@<2.5pt>[rr]^-{X\alpha} \ar@<-2.5pt>[rr]_-{X\beta} \ar[dr]_{(Xu,Xv)} & &
X1^\mathbf 2 \ar[dl]^{(d,c)} \\ & X1 \times X1
}
\end{equation}
exhibits $X2$ as an equivalence relation on $X1^\mathbf 2$ in the slice category $\C / X1 \times X1$;\vskip0.5\baselineskip
\item The graph $(Xu,Xv) \colon X2 \rightrightarrows X1$ bears a (necessarily unique) structure of internal category making the maps $X\alpha$ and $X\beta$ of~\eqref{eq:equiv-relation-diagram} into identity-on-objects internal functors.
\end{enumerate}
$\F_\mathrm{bof}$-congruences are effective in $\cat{Cat}$, and hence in every $\F_\mathrm{bof}$-exact $2$-category.
\end{Prop}
\begin{proof}
As in the proofs of Propositions~\ref{prop:cateads} and~\ref{prop:mixed-cateads}, we may find a set $\Ss$ of $\F_\mathrm{bof}$-congruence axioms such that an object $X \in [\K_\mathrm{bof}, \C]$ inverts $\{h, X\}$ for each $h \in \Ss$ just when it satisfies (a)--(c); and the desired characterisation of the $\F_\mathrm{bof}$-congruences will follow from Propositions~\ref{prop:f-congruence-generating} and~\ref{prop:f-effective-cong} if we can show that every $X \in [\K_\mathrm{bof}, \cat{Cat}]$ satisfying (a)--(c) is in fact the $\F_\mathrm{bof}$-kernel of its own $\F_\mathrm{bof}$-quotient. 

Now, an $X \in [\K_\mathrm{bof}, \cat{Cat}]$ satisfying (a)--(c) determines and is determined by the category $X1$ together with an equivalence relation $\sim$ on each hom-set $X1(x,y)$ which is compatible with composition, in that if $f \sim f' \in X1(x,y)$ and $g \sim g' \in X1(y,z)$, then $gf \sim g'f' \in X1(x,z)$. In these terms, the $\F_\mathrm{bof}$-quotient of $X$ is $q \colon X1 \to Q$, where $Q$ is the category with the same objects as $X1$ and hom-sets $Q(x,y) = X1(x,y) / \mathord\sim$, and $q$ is the evident identity-on-objects quotient functor. On the other hand, the $\F_\mathrm{bof}$-kernel of a functor $F \colon \C \to \D$ is the $X \in [\K_\mathrm{bof}, \cat{Cat}]$ determined by imposing on each $\C(x,y)$ the equivalence relation for which $f \sim f'$ just when $Ff = Ff'$. It is clear from this that every $X \in [\K_\mathrm{bof}, \cat{Cat}]$ satisfying (a)--(c) is the $\F_\mathrm{bof}$-kernel of its $\F_\mathrm{bof}$-quotient as required.
\end{proof}

\section{Elementary descriptions of $2$-dimensional regularity and exactness}\label{sec:elementary}
In this section, we give elementary characterisations of the notions of $\F_\mathrm{so}$-, $\F_\mathrm{bo}$-, and $\F_\mathrm{bof}$-regularity and exactness. We shall do so by applying Theorem~\ref{thm:phi-exact-embedding}; in preparation for which,  we will need to understand how $2$-toposes---here meaning subcategories of a presheaf $2$-category reflective via a finite-limit-preserving reflector---can be constructed from two-dimensional sites. The paper~\cite{Street1982Two-dimensional} gives a thorough treatment of this question; we now summarise the relevant results together with such extensions as will be necessary for our applications.

\subsection{Two-dimensional sheaf theory}
Let $\C$ be a small $2$-category. By a \emph{family} in $\C$, we mean a collection $(f_i \colon U_i \to U \mid i \in I)$ of morphisms with common codomain. Any such family generates a $2$-\emph{sieve} on $U$---a full subobject $m_\mathbf f \colon \phi_\mathbf f \rightarrowtail YU$ in $[\C^\op, \cat{Cat}]$---obtained as the second half of the $\F_\mathrm{so}$-kernel--quotient factorisation
\begin{equation}\label{eq:cotuple-factorisation}
\cd[@!C@-1em]{
\sum_i YU_i \ar[rr]^{\mathbf f \defeq \spn{Yf_i}_{i \in I}} \ar[dr]_{e_\mathbf f} & & YU\rlap{ .}\\
&  \phi_\mathbf f \ar@{ >->}[ur]_{m_\mathbf f}
}
\end{equation}
A presheaf $X \in [\C^\op, \cat{Cat}]$ is said to \emph{satisfy the sheaf condition with respect to $(f_i)$} if it is orthogonal to the associated $2$-sieve $m_\mathbf f$. If $j$ is a collection of families in $\C$, we say that $X$ is a \emph{$j$-sheaf} if it satisfies the sheaf condition with respect to each $(f_i) \in j$, and write $\cat{Sh}_j(\C)$ for the full sub-$2$-category of $[\C^\op, \cat{Cat}]$ on the $j$-sheaves.

By a \emph{$2$-site}, we mean a small, finitely complete $2$-category $\C$ together with a \emph{Grothendieck pretopology} on the underlying ordinary category of $\C$: thus, a collection $j$ of families, called \emph{covering families}, satisfying the following three closure axioms:
\begin{enumerate}
\item[(C)] Given a covering family $(f_i \colon U_i \to U \mid i \in \I)$ and a morphism $g \colon V \to U$, the family $(g^\ast(f_i) \colon V \times_U U_i \to V \mid i \in \I)$ is also covering;\vskip0.5\baselineskip
\item[(M)] For all $U \in \C$,  $(1_U \colon U \to U)$ is covering;\vskip0.5\baselineskip
\item[(L)] Given a covering family $(f_i \colon U_i \to U \mid i \in \I)$, and for each $i\in I$, a covering family $(g_{ik} \colon U_{ik} \to U_i \mid \ell \in I_i)$, the family $(f_i g_{ik} \colon U_{ik} \to U \mid i \in I, k \in I_i)$ is also covering.
\end{enumerate}

\begin{Thm}
For any $2$-site $(\C, j)$, the $2$-category $\cat{Sh}_j(\C)$ is reflective in $[\C^\op, \cat{Cat}]$ via a $2$-functor $L$ which preserves finite limits; in particular, $\cat{Sh}_j(\C)$ is a $2$-topos.
\end{Thm}
\begin{proof}
This is~\cite[Theorem 3.8]{Street1982Two-dimensional}; the notion of $2$-site used there is phrased in terms of Grothendieck topologies (involving sieves) rather than Grothendieck pretopologies (involving covering families), but an examination of the proof shows that it carries over unchanged.
\end{proof}
In practice, we often specify topologies by starting with a class of families satisfying (C), and then closing off under (M) and (L). The following result (which does not appear in~\cite{Street1982Two-dimensional}) shows that, just as in the one-dimensional case, this process does not alter the notion of sheaf.
\begin{Prop}\label{prop:site-reduce}
Let $\C$ be a finitely complete $2$-category, and let $j$ be a collection of covering families satisfying (C). If $\bar \jmath$ denotes the closure of $j$ under (M) and (L), then $\cat{Sh}_j(\C) = \cat{Sh}_{\bar \jmath}(\C)$.
\end{Prop}
\begin{proof}
Let $X \in [\C^\op, \cat{Cat}]$ be a $j$-sheaf, and let $j_X$ denote the class of all families $(f_i)$ such that $X$ satisfies the sheaf axiom with respect to every pullback $(g^\ast f_i)$. We must show that $\bar \jmath \subset j_X$. Clearly $j \subset j_X$, and it is immediate that $j_X$ is closed under (M); it remains to show closure under (L). Thus given $\mathbf f = (f_i \colon U_i \to U \mid i \in I) \in j_X$ and for each $i \in I$, $\mathbf g_i = (g_{ik} \colon U_{ik} \to U_i \mid k \in I_i) \in j_X$, we must show that $\mathbf{fg} = (f_ig_{ik} \mid i \in I, k \in I_i)$ is in $j_X$. Form the commutative diagram
\[
\cd[@R-0.2em]{
\Sigma_{ik} YU_{ik} \ar@{=}[r] \ar@{->>}[d]_{e_{\mathbf g}} &
\Sigma_{ik} YU_{ik} \ar[r]^{\mathbf g} \ar@{->>}[d]_{e_{\mathbf{fg}}} &
\Sigma_{i} YU_i \ar@{->>}[d]^{e_{\mathbf{f}}} \\
\Sigma_i \phi_{\mathbf g_i} \ar@{->>}[r]^u \ar@{ >->}[d]_{m_{\mathbf g}} &
\phi_{\mathbf{fg}} \ar@{ >->}[r]^v \ar@{ >->}[d]_{m_{\mathbf{fg}}} &
\phi_{\mathbf f} \ar@{ >->}[d]^{m_{\mathbf{f}}} \\
\Sigma_{i} YU_i \ar[r]_{\mathbf f}
& YU \ar@{=}[r] 
& YU
}
\]
in $[\C^\op, \cat{Cat}]$, where $\mathbf g = \Sigma_i \mathbf g_i$, $m_\mathbf g = \Sigma_i m_{\mathbf g_i}$, $e_\mathbf g = \Sigma_i e_{\mathbf g_i}$ and $u$ and $v$ are the unique induced maps. The maps marked $\twoheadrightarrow$ and $\rightarrowtail$ are acute and full monic respectively, either by assumption or by the standard properties of orthogonal classes.
Now since $m_\mathbf f \mathbin \bot X$, to show that $m_\mathbf{fg} \mathbin \bot X$, it suffices to show that $v \mathbin \bot X$. Observe that the map $(m_\mathbf g, 1) \colon \mathbf f.m_\mathbf g \to \mathbf f$ in $[\C^\op, \cat{Cat}]^\mathbf 2$ gives a map of $\F_\mathrm{so}$-kernels:
%
\[
\cd[@R-1.5em@C-1em]{
& R[\mathbf{f}m_{\mathbf g}] \ar[dr]^{j} \ar[dd]^(0.25)y
\\
\mathbf{f}m_{\mathbf g} \mid \mathbf{f}m_{\mathbf g} \mid \mathbf{f}m_{\mathbf g} \ar@<6pt>[rr]^(0.4){p} \ar[rr]|(0.4){m} \ar@<-6pt>[rr]_(0.4){q} \ar[dd]^{x} & &
\mathbf{f}m_{\mathbf g} \mid \mathbf{f}m_{\mathbf g} \ar@<6pt>[rr]^-{d} \ar@{<-}[rr]|-{i} \ar@<-6pt>[rr]_-{c} \ar[dd]^{w} & &
\Sigma_{i}\phi_{\mathbf{g}_i} \ar[dd]^{m_\mathbf g}\\
& R[\mathbf{f}] \ar[dr]^j 
\\
\mathbf{f} \mid \mathbf{f} \mid \mathbf{f} \ar@<6pt>[rr]^(0.4){p} \ar[rr]|(0.4){m} \ar@<-6pt>[rr]_(0.4){q} & &
\mathbf{f} \mid \mathbf{f} \ar@<6pt>[rr]^-{d} \ar@{<-}[rr]|-{i} \ar@<-6pt>[rr]_-{c} & &
\Sigma_{i}YU_{i} \rlap{ .}}
\]
As $[\C^\op, \cat{Cat}]$ is $\F_\mathrm{so}$-regular, the $\F_\mathrm{so}$-quotients of the two rows are $\phi_\mathrm{fg}$ and $\phi_\mathrm f$ respectively, and it's easy to see that the induced comparison map is  $v \colon \phi_{\mathrm{fg}} \to \phi_f$. Since the class of maps orthogonal to $X$ is closed under colimits, we may conclude that $v \mathbin \bot X$ so long as $\{m_\mathbf g, w, y, z\} \mathbin \bot X$. Now $m_\mathbf g = \Sigma_i {m_{\mathbf g_i}}$ is orthogonal to $X$ because each $m_{\mathbf g_i}$ is so; as for $w$, we can write it as a composite $w = w_2w_1$ of pullbacks:
\[
\cd{
\mathbf{f}m_\mathbf{g} \mid \mathbf{f}m_\mathbf{g} \ar[r]^d \ar[d]_{w_1} & \Sigma_i \phi_{\mathbf g_i}\ar[d]^{m_\mathbf g} \\
\mathbf{f} \mid \mathbf{f}m_\mathbf{g} \ar[r]_{d} & \Sigma_i YU_i} \qquad \text{and} \qquad
\cd{\mathbf{f} \mid \mathbf{f}m_\mathbf{g} \ar[r]^{c} \ar[d]_{w_2} & \Sigma_i \phi_{\mathbf g_i}\ar[d]^{m_\mathbf g} \\
\mathbf{f} \mid \mathbf{f} \ar[r]_c & \Sigma_i YU_i\rlap{ .}} 
\]
Considering $w_2$, we see that $\mathbf f \mid \mathbf f \cong \Sigma_{i, i'} Y(f_i \mid f_{i'})$; it follows that $w_2$ is a coproduct $\Sigma_{i,i'} m_{\mathbf h_{i,i'}}$ with each $\mathbf h_{i,i'}$ a pullback of some $\mathbf g_k$. Since each $\mathbf g_k \in j_X$, also each $\mathbf h_{i,i'} \in j_X$, and so 
each $m_{\mathbf h_{i,i'}} \mathbin \bot X$. Thus $w_2 \mathbin \bot X$ and similarly $w_1 \bot X$, whence $w \mathbin \bot X$ as required. Corresponding arguments show that $\{y,z\} \mathbin \bot X$.
\end{proof}

The following result is stated only for \emph{singleton} covers merely as a convenience; there is a corresponding version for general covering families---harder to state, though scarcely harder to prove---but we will not need it in this paper.
\begin{Prop}\label{prop:yoneda-makes-mixed-codescent}
For any $2$-site $(\C, j)$, the composite
\begin{equation*}
\C \xrightarrow Y [\C^\op, \cat{Cat}] \xrightarrow L \cat{Sh}_j(\C)
\end{equation*}
sends singleton covers to effective $\F_\mathrm{so}$-quotient maps.
\end{Prop}

\begin{proof}
For a singleton cover $f \colon V \to U$ in $\C$, the factorisation~\eqref{eq:cotuple-factorisation} is an $\F_\mathrm{so}$-kernel--quotient factorisation of $Yf$ in $[\C^\op, \cat{Cat}]$; since $L$ preserves small colimits and finite limits, $LYf = Le_f.Lm_f$ is also an $\F_\mathrm{so}$-kernel--quotient factorisation; but since $m_f$ is a covering $2$-sieve, $Lm_f$ is invertible, whence $LYf$ is an effective $\F_\mathrm{so}$-quotient.
%
%
%
%
%
\end{proof}

A $2$-site is called \emph{subcanonical} if every representable $2$-functor is a $j$-sheaf. For such a $2$-site, the functor $LY \colon \C \to \cat{Sh}_j(\C)$ is actually a factorisation of the Yoneda embedding through the full inclusion $\cat{Sh}_j(\C) \rightarrowtail [\C^\op, \cat{Cat}]$, and, as such, is fully faithful.
A $2$-site is subcanonical just when for every covering $2$-sieve $m \colon \phi \rightarrowtail YU$, we have isomorphisms of categories
\begin{equation*}
\C(U,K) \xrightarrow{Y} [\C^\op, \cat{Cat}](YU, YK) \xrightarrow{(\thg) \circ m} [\C^\op, \cat{Cat}](\phi, YK)
\end{equation*}
for every $K \in \C$; which is to say that $m$ exhibits $U$ as the colimit $\phi \star 1_\C$. 

\begin{Prop}\label{prop:cover-singleton-subcanonical}
Let $\J$ be a class of maps in the finitely complete $2$-category $\C$, and let $j$ be the smallest Grothendieck pretopology on $\C_0$ for which each $f \in \J$ is  a singleton cover. Then the $2$-site $(\C, j)$ is subcanonical if and only if $\J$ is composed of pullback-stable effective $\F_\mathrm{so}$-quotient maps.
\end{Prop}
\begin{proof}
In one direction, if $(\C, j)$ is subcanonical, then by the preceding result, the restricted Yoneda embedding $Y \colon \C \to \cat{Sh}_j(\C)$ sends every map in $\J$ to a stable effective $\F_\mathrm{so}$-quotient; but as $Y$ preserves finite limits and is fully faithful, it reflects stable effective $\F_\mathrm{so}$-quotients, and so $\J$ is composed solely of such maps.

In the converse direction, it suffices by Proposition~\ref{prop:site-reduce} to check that representables satisfy the sheaf condition for $(f)$ whenever $f \colon V \to U$ is a pullback of a map in $\J$.
In this situation, the induced $2$-sieve $m_f \colon \phi_f \rightarrowtail YU$ is as before the second half of the $\F_\mathrm{so}$-kernel--quotient factorisation of $Yf \colon YV \to YU$ in $[\C^\op, \cat{Cat}]$; and so $\phi_f$ is the $\F_\mathrm{so}$-quotient of $K(Yf) \cong Y(Kf)$. Since taking weighted colimits is cocontinuous in the weight insofar as it is defined, and colimits by representable weights are given by evaluation at the representing object, we conclude that the colimit $\phi_f \star 1_\C$, if it exists, must be the $\F_\mathrm{so}$-quotient of the $\F_\mathrm{so}$-kernel of $f$. Thus to say that $m_f$ exhibits $U$ as $\phi_f \star 1_\C$ is to say that $f$ is the $\F_\mathrm{so}$-quotient of its own $\F_\mathrm{so}$-kernel, that is, an effective $\F_\mathrm{so}$-quotient map, which is so by assumption.
\end{proof}


\subsection{(Surjective on objects, injective on objects and fully faithful)}

\begin{Thm}\label{thm:so-regular}
A $2$-category $\C$ with finite limits and $\F_\mathrm{so}$-quotients of $\F_\mathrm{so}$-kernels is $\F_\mathrm{so}$-regular if and only if $\F_\mathrm{so}$-quotient maps are effective and stable under pullback.
\end{Thm}
\begin{proof}
The ``only if'' direction is contained in Corollary~\ref{cor:ker-quot-so-2topos}. For the ``if'' direction, it suffices by Proposition~\ref{prop:phi-small-reduction} and the remarks following, to prove that a \emph{small} $\C$ in which $\F_\mathrm{so}$-quotient maps are effective and stable is $\F_\mathrm{so}$-regular. By Theorem~\ref{thm:phi-exact-embedding}, it suffices to exhibit a full embedding into a $2$-topos which preserves finite limits and $\F_\mathrm{so}$-quotients of $\F_\mathrm{so}$-kernels. So consider on $\C$ the \emph{$\F_\mathrm{so}$-regular topology} generated by taking all $\F_\mathrm{so}$-quotient maps as singleton covers. By assumption, every such $f$ is effective and stable under pullback, and thus by Proposition~\ref{prop:cover-singleton-subcanonical}, the topology they generate makes $\C$ into a subcanonical $2$-site. We thus have a fully faithful embedding $\C \to \cat{Sh}(\C)$ which preserves all limits; it remains to show that it preserves $\F_\mathrm{so}$-quotients of $\F_\mathrm{so}$-kernels. Since $\F_\mathrm{so}$-quotient maps are effective in $\C$ (by assumption) and in $\cat{Sh}(\C)$ (since it is a $2$-topos), it suffices by Proposition~\ref{prop:preserve-quotients-preserve-colims}(a) to show that $\C \to \cat{Sh}(\C)$ preserves $\F_\mathrm{so}$-quotient maps; but this is so by the definition of $j$, Proposition~\ref{prop:yoneda-makes-mixed-codescent} and Corollary~\ref{cor:ker-quot-so-2topos}.
\end{proof}
\begin{Rk}
In~\cite[\S 1.19]{Street1982Two-dimensional}, a finitely complete $2$-category is defined to be regular if each morphism admits an (acute, full monic) factorisation, and acute morphisms are stable under pullback. Theorem~1.22 of~\cite{Street1982Two-dimensional} claims that, in any such $2$-category, each acute morphism is an effective $\F_\mathrm{so}$-quotient; given which, Street's definition would coincide with ours. Unfortunately, the proof of Theorem~1.22 contains an error\footnote{The erroneous sentence reads ``Since $f = us$, it follows that $\mathbf E(s) \cong \mathbf E(f)$''.} and the result is in fact false. To see this, observe that in the $2$-category $\cat{Ab}\text-\cat{Cat}$, factorising an $\cat{Ab}$-functor through its full image yields pullback-stable (acute, full monic) factorisations, but that by Proposition~\ref{prop:mixed-codescent-not-effective}, not every acute map in $\cat{Ab}\text-\cat{Cat}$ is an effective $\F_\mathrm{so}$-quotient.
\end{Rk}

\begin{Thm}
A $2$-category $\C$ with finite limits and $\F_\mathrm{so}$-quotients of $\F_\mathrm{so}$-congruences is $\F_\mathrm{so}$-exact if and only if it is $\F_\mathrm{so}$-regular and $\F_\mathrm{so}$-congruences are effective.
\end{Thm}
\begin{proof}
The ``only if'' direction is contained in Propositions~\ref{prop:f-exact-f-regular} and~\ref{prop:mixed-cateads}. The ``if'' direction is argued as in the preceding result, taking sheaves again for the $\F_\mathrm{so}$-regular topology but now using part (b) rather than part (a) of Proposition~\ref{prop:preserve-quotients-preserve-colims}.
\end{proof}
\begin{Rk}
Modulo the discrepancy noted in the preceding remark, this agrees with the definition of exact $2$-category in~\cite[\S 2.1]{Street1982Two-dimensional}.
\end{Rk}

\subsection{(Bijective on objects, fully faithful)}
\begin{Lemma}\label{lemma:eff-codesc-eff-mixed}
In any finitely complete $2$-category, effective $\F_\mathrm{bo}$-quotient maps are effective $\F_\mathrm{so}$-quotient maps.
\end{Lemma}
\begin{proof}
Let $\C$ be a $2$-category with finite limits, and $f \colon A \to B$ an effective $\F_\mathrm{bo}$-quotient map. In the presheaf $2$-category $[\C^\op, \cat{Cat}]$, we may factorise $Yf \colon YA \to YB$ as
\begin{equation*}
YA \xrightarrow{g_1} \phi_1 \xrightarrow{g_2} \phi_2 \xrightarrow{g_3} YB
\end{equation*}
where $g_1$ is pointwise bijective on objects, $g_2$ is pointwise surjective on objects and fully faithful, and $g_3$ is pointwise injective on objects and fully faithful. By the argument of Proposition~\ref{prop:cover-singleton-subcanonical}, $f$ is an effective $\F_\mathrm{so}$-quotient map if and only if $g_3 \mathbin \bot YK$ for each $K \in \K$; analogously,  $f$ is an effective $\F_\mathrm{bo}$-quotient map if and only if $g_3 g_2 \mathbin \bot YK$ for each $K \in \K$. But since $g_2$ is epimorphic, $g_3 g_2 \mathbin \bot YK$ implies $g_3 \mathbin \bot YK$ by the standard cancellativity properties of orthogonality classes.
\end{proof}
\begin{Prop}\label{prop:codescent-in-boregular}
Given $f \colon A \to B$ in an $\F_\mathrm{bo}$-regular $2$-category $\C$, the following are equivalent:
\begin{enumerate}[(i)]
\item $f$ is an effective $\F_\mathrm{bo}$-quotient map;
\item $f$ and $\delta_f \colon A \to A \times_B A$ are effective $\F_\mathrm{bo}$-quotient maps;
\item $f$ and $\delta_f$ are stable effective $\F_\mathrm{so}$-quotient maps;
\item $f$ and $\delta_f$ are acute.
\end{enumerate}
\end{Prop}
\begin{proof}
(i) $\Rightarrow$ (ii) is clear when $\C = \cat{Cat}$, and hence also when $\C = \P\B$, since limits and colimits in $\P \B$ are pointwise. Consider next the case $\C = \Phi_\mathrm{bo}(\B)$.  Because $\Phi_\mathrm{bo}(\B)$ is closed in $\P \B$ under finite limits and $\F_\mathrm{bo}$-quotients of $\F_\mathrm{bo}$-kernels, it follows from Proposition~\ref{prop:preserve-quotients-preserve-colims} that any effective $\F_\mathrm{bo}$-quotient $f \in \Phi_\mathrm{bo}(\B)$ remains such in $\P \B$. Thus $\delta_f$ is an effective $\F_\mathrm{bo}$-quotient map in $\P \B$, whence also in the full subcategory $\Phi_\mathrm{bo}(\B)$ as required. Finally, let $f$ be an effective $\F_\mathrm{bo}$-quotient map in the arbitrary $\F_\mathrm{bo}$-regular $\C$. Take its image under the embedding $Z \colon \C \to \Phi_\mathrm{bo}(\C)$, and form the kernel--quotient factorisation for $\F_\mathrm{bo}$, given by $Zf = hg$, say. Since $g$ is an effective $\F_\mathrm{bo}$-quotient map in $\Phi_\mathrm{bo}(\C)$, so is $\delta_g$; now applying $L$ to the factorisation $Zf = hg$ yields the corresponding factorisation in $\C$, whence $Lh$ is invertible and so $f \cong Lg$. But now $\delta_f \cong \delta_{Lg} \cong L(\delta_g)$, like $\delta_g$, is an effective $\F_\mathrm{bo}$-quotient map as required.

This proves (i) $\Rightarrow$ (ii); now (ii) $\Rightarrow$ (iii) is Lemma~\ref{lemma:eff-codesc-eff-mixed} together with Corollary~\ref{cor:ker-quot-bo-2topos}, (iii) $\Rightarrow$ (iv) is Proposition~\ref{prop:f-map-props}(b), and it remains to prove (iv) $\Rightarrow$ (i).
So suppose that $f \colon A \to B$ and $\delta_f \colon A \to A \times_B A$ are both acute; we must show that $f$ is an effective $\F_\mathrm{bo}$-quotient map. Thus, on forming the $\F_\mathrm{bo}$-kernel--quotient factorisation
\begin{equation*}
f = A \xrightarrow e C \xrightarrow m B\rlap{ ,}
\end{equation*}
we must verify that $m$ is an isomorphism. $f$ is assumed acute, and $e$ is an $\F_\mathrm{bo}$-quotient, hence an $\F_\mathrm{bo}$-strong epi, hence acute; thus $m$ is also acute by Proposition~\ref{prop:f-map-props}(d). But $m$ is also fully faithful, because by Corollary~\ref{cor:ker-quot-bo-2topos}, kernel--quotient factorisations in $\C$ converge immediately. It is now enough to show that $m \colon C \to B$ is monic; for then it will be full monic and acute, whence invertible.
To show monicity is equally to show that the diagonal $\delta_{m} \colon C \to C \times_{B} C$ is invertible. Since $m$ is faithful, it follows easily that $\delta_{m}$ is full monic; so it is enough to show that $\delta_m$ is also acute. Consider the square
\begin{equation*}
\cd{
A \ar[r]^-{\delta_{f}} \ar[d]_e & A \times_{B} A \ar[d]^{e \times_{B} e} \\
C \ar[r]_-{\delta_m} & C \times_{B} C\rlap{ .}
}
\end{equation*}
In it, $e$ is an $\F_\mathrm{bo}$-quotient map, whence also the pullback $e \times_{B} A$, since $\F_\mathrm{bo}$-quotient maps in $\C$ are stable by Corollary~\ref{cor:ker-quot-bo-2topos}; 
similarly the pullback $C \times_{B} e$ is an $\F_\mathrm{bo}$-quotient map. Hence $e \times_{B} e = (C \times_{B} e) \circ (e \times_{B} A)$ is a composite of acute maps, and so acute.
$\delta_f$ is acute by assumption, and so the common diagonal of the square is acute; since $e$ is acute, it follows by Proposition~\ref{prop:f-map-props}(d) that $\delta_m$ is acute as required.
\end{proof}
\begin{Rk}\label{rk:geometric}
The correspondence between (i) and (iii) in this proposition can be understood in terms of ``geometric $2$-logic''. The statement that a morphism $f$ in a $2$-category $\C$ is a stable effective $\F_\mathrm{so}$-quotient map can be interpreted as saying that, in the internal logic of $\C$ equipped with its canonical topology, $f$ is surjective on objects. Correspondingly, the statement that $\delta_f$ be a stable effective $\F_\mathrm{so}$-quotient can be interpreted as saying that $f$ is injective on objects. Thus the equivalence of (i) and (iii) says that in a $\F_\mathrm{bo}$-regular $2$-category, ``a map is an $\F_\mathrm{bo}$-quotient if and only if it is surjective on objects and injective on objects''.
\end{Rk}

\begin{Thm}\label{thm:bo-char}
A $2$-category with finite limits and $\F_\mathrm{bo}$-quotients of $\F_\mathrm{bo}$-kernels is $\F_\mathrm{bo}$-regular just when $\F_\mathrm{bo}$-quotient maps are effective and stable under pullback, and whenever $f \colon A \to B$ is an $\F_\mathrm{bo}$-quotient map, so also is $\delta_f \colon A \to A \times_B A$.
\end{Thm}
\begin{proof}
The ``only if'' direction is contained in Corollary~\ref{cor:ker-quot-bo-2topos} and Proposition~\ref{prop:codescent-in-boregular}. 
For the ``if'' direction, it suffices by Proposition~\ref{prop:phi-small-reduction} and the remarks following, to prove it only for a \emph{small} $\C$ satisfying the stated hypotheses. By Theorem~\ref{thm:phi-exact-embedding}, it suffices to exhibit a full embedding of $\C$ into a $2$-topos which preserves finite limits and  $\F_\mathrm{bo}$-quotients of $\F_\mathrm{bo}$-kernels. Consider on $\C$ the \emph{$\F_\mathrm{bo}$-regular topology} generated by taking every $\F_\mathrm{bo}$-quotient map as a singleton cover. By assumption, these covers are stable under pullback, and are effective $\F_\mathrm{so}$-quotients by Lemma~\ref{lemma:eff-codesc-eff-mixed}; thus the topology they generate makes $\C$ into a subcanonical $2$-site. So we obtain a fully faithful embedding $\C \to \cat{Sh}(\C)$ which preserves all limits; it remains to show that it preserves  $\F_\mathrm{bo}$-quotients of $\F_\mathrm{bo}$-kernels. Since  $\F_\mathrm{bo}$-quotient maps are effective in $\C$ (by assumption) and in $\cat{Sh}(\C)$ (since it is a $2$-topos), it suffices by Proposition~\ref{prop:preserve-quotients-preserve-colims}(a) to show that $Y \colon \C \to \cat{Sh}(\C)$ preserves $\F_\mathrm{bo}$-quotient maps. But if $f$ is an $\F_\mathrm{bo}$-quotient map in $\C$, then so also is $\delta_f$, and so by the definition of $j$ and Proposition~\ref{prop:yoneda-makes-mixed-codescent}, both $Yf$ and $Y\delta_f \cong \delta_{Yf}$ are acute in the $2$-topos $\cat{Sh}(\C)$; whence, by Proposition~\ref{prop:codescent-in-boregular}, $Yf$ is an $\F_\mathrm{bo}$-quotient map.
\end{proof}

\begin{Rk}
The condition that $\delta_f$ be an $\F_\mathrm{bo}$-quotient whenever $f$ is so is substantive. Indeed, if we view $\cat{Set}$ as a locally discrete $2$-category, then $\F_\mathrm{bo}$-kernels therein are simply extended kernel-pair diagrams, and the quotient of $X \in [\K_\mathrm{bo}, \cat{Set}]$ is simply the coequaliser of $Xd,Xc \colon X2 \rightrightarrows X1$. It follows that $\F_\mathrm{bo}$-quotients are regular epimorphisms; as such, they are effective and stable under pullback. However, if $f \colon A \to B$ is a regular epimorphism, then the diagonal $\delta_f \colon A \to A \times_B A$ cannot be so unless $f$ is actually invertible. Thus $\cat{Set}$, which satisfies all the other hypotheses for $\F_\mathrm{bo}$-regularity, does not verify this one. We can understand this failure in terms of Remark~\ref{rk:geometric}: in the internal geometric $2$-logic of $\cat{Set}$, ``$\F_\mathrm{bo}$-quotients are surjective on objects, but not injective on objects''.
\end{Rk}

\begin{Thm}\label{thm:boex-char}
A $2$-category $\C$ with finite limits and $\F_\mathrm{bo}$-quotients of  $\F_\mathrm{bo}$-congruences is $\F_\mathrm{bo}$-exact just when it is $\F_\mathrm{bo}$-regular and $\F_\mathrm{bo}$-congruences in $\C$ are effective.
\end{Thm}
\begin{proof}
The ``only if'' direction is contained in Propositions~\ref{prop:f-exact-f-regular} and~\ref{prop:cateads}. The ``if'' direction is argued as in the preceding result, again taking sheaves for the $\F_\mathrm{bo}$-regular topology but now using part (b) rather than part (a) of Proposition~\ref{prop:preserve-quotients-preserve-colims}.
\end{proof}

\subsection{(Bijective on objects and full, faithful)}

\begin{Prop}\label{prop:coeq-in-bofregular}
If $f \colon A \to B$ is a morphism in an $\F_\mathrm{bof}$-regular $2$-category $\C$, then the following are equivalent:
\begin{enumerate}[(i)]
\item $f$ is an (effective) $\F_\mathrm{bof}$-quotient map;
\item $f$, $\delta_f \colon A \to A \times_B A$ and $\gamma_f \colon A^\mathbf 2 \to f \mathord{\mid} f$ are stable effective $\F_\mathrm{so}$-quotient maps;
\item $f$, $\delta_f$ and $\gamma_f$ are acute.
\end{enumerate}
\end{Prop}
\begin{proof}
The proof of (i) $\Rightarrow$ (ii) is identical in form to that in Proposition~\ref{prop:codescent-in-boregular}; (ii) $\Rightarrow$ (iii) is, again, Proposition~\ref{prop:f-map-props}(b); it remains to prove (iii) $\Rightarrow$ (i). So suppose that $f \colon A \to B$,  $\delta_f \colon A \to A \times_B A$ and $\gamma_f \colon A^\mathbf 2 \to f \mathord{\mid} f$ are all acute; we must show that $f$ is an $\F_\mathrm{bof}$-quotient map. Forming the kernel--quotient factorisation
\begin{equation*}
f = A \xrightarrow e C \xrightarrow m B
\end{equation*}
for $\F_\mathrm{bof}$, we must show that $m$ is an isomorphism. We argue as before: since $f$ is acute so is $m$, and so it suffices to show that $m$ is full monic. Certainly, it is faithful, because by Corollary~\ref{cor:ker-quot-bof-2topos}, kernel--quotient factorisations in $\C$ converge immediately. Now the argument of Proposition~\ref{prop:codescent-in-boregular} shows that $m$ is monic; it remains to show that it is fully faithful. To do so is equally to show that $\gamma_{m} \colon C^\mathbf 2 \to m \mathord{\mid} m$ is invertible. Since $m$ is faithful, it follows easily that $\gamma_{m}$ is full monic; so it's enough to show that it is also acute. Consider the commutative diagram:
\begin{equation*}
\cd{
A^\mathbf 2 \ar[r]^-{\gamma_{f}} \ar[d]_{e^\mathbf 2} & f \mathord{\mid} f \ar[d]^{e \mathord{\mid} e} \ar[r] & A \times A \ar[d]^{e \times e} \\
C^\mathbf 2 \ar[r]_-{\gamma_m} & m \mathord{\mid} m \ar[r] & C \times C\rlap{ .}
}
\end{equation*}
Since $e$ is an $\F_\mathrm{bof}$-quotient map, and such morphisms  are stable in $\C$ by Corollary~\ref{cor:ker-quot-bof-2topos}, both the pullbacks $e \times A$ and $C \times e$
are $\F_\mathrm{bof}$-quotient maps. Thus $e \times e$ is a composite of $\F_\mathrm{bof}$-quotient maps; since the right-hand square is a pullback, $e \mathord{\mid} e$ is also a composite of $\F_\mathrm{bof}$-quotient maps, and as such is acute. $\gamma_f$ is acute by assumption, and so the common diagonal of the square is acute; thus, by Proposition~\ref{prop:strongcanc}, $\gamma_m$ is acute as required.
\end{proof}
\begin{Rk}
As in Remark~\ref{rk:geometric}, the equivalence of (i) and (ii) can be interpreted in terms of geometric $2$-logic. The statement that $\gamma_f$ be a stable effective $\F_\mathrm{so}$-quotient is the statement that the internal logic of $\C$ sees $f$ as full. Thus the equivalence of (i) and (ii) says that ``$\F_\mathrm{bof}$-quotients are precisely the maps which are surjective on objects, injective on objects, and full''.
\end{Rk}

\begin{Thm}
A $2$-category $\C$ with finite limits and $\F_\mathrm{bof}$-quotients of $\F_\mathrm{bof}$-kernels is $\F_\mathrm{bof}$-regular if and only if, whenever $f \colon A \to B$ is an $\F_\mathrm{bof}$-quotient map in $\C$, each of the maps
\begin{equation}\label{eq:singleton}
f \colon A \to B \qquad \text{and} \qquad \delta_f \colon A \to A \times_B A \qquad \text{and} \qquad \gamma_f \colon A^\mathbf 2 \to f \mathord{\mid} f
\end{equation}
is a pullback-stable effective $\F_\mathrm{so}$-quotient map.
\end{Thm}
\begin{proof}
The ``only if'' direction is contained in Corollary~\ref{cor:ker-quot-bof-2topos} and Proposition~\ref{prop:coeq-in-bofregular}. 
For the ``if'' direction, it suffices by Proposition~\ref{prop:phi-small-reduction} and the remarks following, to prove it only for a \emph{small} $\C$ satisfying the stated hypotheses. By Theorem~\ref{thm:phi-exact-embedding}, it suffices to exhibit a full embedding of $\C$ into a $2$-topos which preserves finite limits and  $\F_\mathrm{bof}$-quotients of $\F_\mathrm{bof}$-kernels.  Consider on $\C$ the \emph{$\F_\mathrm{bof}$-regular topology} generated by taking every morphism of the form $f$, $\delta_f$ and $\gamma_f$, for $f$ an $\F_\mathrm{bof}$-quotient map, as a singleton cover. By assumption, these covers are stable effective $\F_\mathrm{so}$-quotient maps, and so the topology they generate makes $\C$ into a subcanonical $2$-site. So we obtain a fully faithful embedding $\C \to \cat{Sh}(\C)$ which preserves all limits; it remains to show that it preserves  $\F_\mathrm{bof}$-quotients of $\F_\mathrm{bof}$-kernels. Since  $\F_\mathrm{bof}$-quotient maps are always effective, it suffices by Proposition~\ref{prop:preserve-quotients-preserve-colims}(a) to show that $Y \colon \C \to \cat{Sh}(\C)$ preserves $\F_\mathrm{bof}$-quotient maps. But if $f$ is such a map in $\C$, then by the definition of $j$ and Proposition~\ref{prop:yoneda-makes-mixed-codescent}, each of $Yf$, $Y\delta_f \cong \delta_{Yf}$ and $Y(\gamma_f) \cong \gamma_{Yf}$ is acute; and so by Proposition~\ref{prop:codescent-in-boregular}, $Yf$ is an $\F_\mathrm{bof}$-quotient as required.
\end{proof}
\begin{Rk}
The hypotheses of this theorem are substantive. Consider in the $2$-category $\cat{Ab}\text-\cat{Cat}$ the morphism $\phi \colon \F_2 \to 1$, where $1$ is the terminal $\cat{Ab}$-category and $\F_2$ the one-object $\cat{Ab}$-category on the field $\mathbb F_2$. This $\phi$ is an $\F_\mathrm{bof}$-quotient map, being the coequifier of the two possible $2$-cells $\alpha, \beta \colon \psi \Rightarrow \psi \colon \I \to \F_2$ (with $\I$ the unit $\cat{Ab}$-category and $\psi$ the unique $\cat{Ab}$-functor into $\F_2$). Tracing through the argument of Proposition~\ref{prop:mixed-codescent-not-effective}, we see that, although $\phi$ is an effective $\F_\mathrm{so}$-quotient, the diagonal of its kernel-pair $\F_2 \to \F_2 \times \F_2$ is not so.
\end{Rk}

\begin{Thm}
A $2$-category $\C$ with finite limits and $\F_\mathrm{bof}$-quotients of $\F_\mathrm{bof}$-congruences is $\F_\mathrm{bof}$-exact if and only if it is $\F_\mathrm{bof}$-regular and $\F_\mathrm{bof}$-congruences are effective in $\C$.
\end{Thm}
\begin{proof}
The ``only if'' direction is contained in Propositions~\ref{prop:f-exact-f-regular} and~\ref{prop:congruences}. The ``if'' direction is argued as in the preceding result, taking again sheaves for the $\F_\mathrm{bof}$-regular topology, but now using part (b) rather than part (a) of Proposition~\ref{prop:preserve-quotients-preserve-colims}.
\end{proof}

\section{Relationships between the notions}\label{sec:relationships}
In the following section, we will give a range of examples of the regularity and exactness notions introduced above; and in order to do so efficiently, it will be useful to study the interrelations between them. We begin by relating $\F_\mathrm{so}$-regularity and $\F_\mathrm{bo}$-regularity. There is no direct implication; as we shall see in the following section, $\cat{Set}$, seen as a locally discrete $2$-category is $\F_\mathrm{so}$-regular but not $\F_\mathrm{bo}$-regular, whilst if $\E$ is a $1$-category which is not regular, then $\cat{Cat}(\E)$ is $\F_\mathrm{bo}$-regular but not $\F_\mathrm{so}$-regular. However, there is something we can say. Let us call a finitely complete $2$-category \emph{ff-regular} if the kernel-pair of any fully faithful map admits a fully faithful coequaliser, and fully faithful regular epis are stable under pullback.

\begin{Prop}\label{prop:reg-ff}
Any $\F_\mathrm{so}$-regular $2$-category is ff-regular; any $\F_\mathrm{bo}$-regular and ff-regular $2$-category is $\F_\mathrm{so}$-regular.
\end{Prop}
The following lemma will be crucial to the proof:
\begin{Lemma}\label{lemma:eqrel-lemma}
Let $(s,t) \colon E \rightrightarrows A$ be an equivalence relation in a $2$-category, with reflexivity $r \colon A \to E$, say.
If $s$ (equivalently, $t$) is fully faithful, there is a unique invertible $2$-cell $\theta \colon s \Rightarrow t$ such that $\theta r = 1_{1_A}$.
The map $j \colon E \to A^\mathbf 2$ induced by $\theta$ is part of an $\F_\mathrm{so}$-congruence
\begin{equation}\label{eq:ff-kernel}
\cd{
& E \ar[d]_j
\\
A^\mathbf 3 \ar@<6pt>[r]^-{p} \ar[r]|-{m} \ar@<-6pt>[r]_-{q} &
A^\mathbf 2 \ar@<6pt>[r]^-{d} \ar@{<-}[r]|-{i} \ar@<-6pt>[r]_-{c} &
A
}
\end{equation}
and the following colimits, if existing, coincide: the $\F_\mathrm{so}$-quotient of~\eqref{eq:ff-kernel}; the coidentifier of $\theta \colon s \Rightarrow t$; and the coequaliser of $(s,t)$.
\end{Lemma}
\begin{proof}
We have $sr = 1_A$, whence $1_s \colon srs \Rightarrow s1_E$, and so by full fidelity of $s$ an  invertible $2$-cell $rs \cong 1$; thus $r$ is an equivalence, and so from the $2$-cell $1_{1_A} \colon sr \Rightarrow tr$ we deduce the presence of a unique $2$-cell $\theta \colon s \Rightarrow t$ with $\theta r = 1_{1_A}$.
We now show that the induced $j \colon E \to A^\mathbf 2$ makes~\eqref{eq:ff-kernel} an $\F_\mathrm{so}$-congruence. First, $j$ is monic because $(s,t)$ is monic and $(d,c)j = (s,t)$. To see that it is fully faithful, note that it factors through the fully faithful $h \colon A^{\mathbf I} \to A^\mathbf 2$ (where $\mathbf I$ is the free category on an isomorphism); and now
in the decomposition
\[
s = E \xrightarrow {k} A^{\mathbf I} \xrightarrow{h} A^\mathbf 2 \xrightarrow{d} A
\]
we have $s$ fully faithful by assumption, and $dh$ an equivalence, hence fully faithful; whence $k$ is fully faithful, and so also $hk = j$ as desired. 
Finally, we check that $j$ induces a bijective-on-objects internal functor $(E \rightrightarrows A) \to (A^\mathbf 2 \rightrightarrows A)$. The equality $\theta r = 1_{1_A}$ shows that it preserves identities; as for composition, we must show that
\[
\cd{
E \times_A E \ar[r]_-{c} & E \ar@/^8pt/[r]^{s} \dtwocell{r}{\theta} \ar@/_8pt/[r]_{t}  & A
} \qquad = \qquad
\cd[@C+0.5em@R-1em]{
& E \ar@/^8pt/[dr]^{s} \dtwocell{dr}{\theta} \ar@/_8pt/[dr]_(0.4){t}  \\
E \times_A E \ar[ur]^-{\pi_1} \ar[dr]_-{\pi_2} & & A \\
& E \ar@/^8pt/[ur]^(0.4){s} \dtwocell{ur}{\theta} \ar@/_8pt/[ur]_{t}
}\]
holds (where $c$ witnesses the transitivity of $E$). This is trivially true on precomposing with $(r,r) \colon A \to E \times_A E$, so we will be done so long as $(r,r)$ is an equivalence. But $\pi_1.(r,r) = r \colon A \to E$ is an equivalence, and so is $\pi_1 \colon E \times_A E \to E$, as it is the pullback of the surjective equivalence $s$ along $t$. Thus by two-out-of-three $(r,r)$ is an equivalence and so the displayed equality obtains.

For the final sentence of the proposition, note that, in any $2$-category, the $\F_\mathrm{bo}$-congruence along the bottom of~\eqref{eq:ff-kernel} has $1_A \colon A \to A$ as its quotient; whence the $\F_\mathrm{so}$-quotient of~\eqref{eq:ff-kernel} is equally the coidentifier of $\theta \colon s \Rightarrow t$. To show that this is in turn the same as the coequaliser of $s$ and $t$, we must show that any $f \colon A \to B$ with $fs = ft$ also has $f\theta = 1_{fs}$. But $f\theta r = f1_{1_A} = 1_{fs} r$ whence $f\theta = 1_{fs}$.
\end{proof}

We are now ready to give:

\begin{proof}[Proof of Proposition~\ref{prop:reg-ff}]
Suppose first that $\C$ is $\F_\mathrm{so}$-regular. If $f \colon A \to B$ is fully faithful in $\C$, then $f \mathord \mid f \cong A^\mathbf 2$ and $f \mathord \mid f \mathord \mid f \cong A^\mathbf 3$, whence the $\F_\mathrm{so}$-kernel of $f$ is a congruence of the form~\eqref{eq:ff-kernel}, with $E$ the kernel-pair of $f$. So by Lemma~\ref{lemma:eqrel-lemma}, the $\F_\mathrm{so}$-quotient $q \colon A \to Q$ of this congruence is equally the coequaliser of $f$'s kernel-pair; moreover, as $\F_\mathrm{so}$-kernels are effective in $\C$, we have $q \mathord \mid q \cong A^\mathbf 2$ and so $q$ is fully faithful. Finally, fully faithful regular epis are pullback-stable because, by Lemma~\ref{lemma:eqrel-lemma}, they are exactly the fully faithful effective $\F_\mathrm{so}$-quotients, and so stable since $\C$ is $\F_\mathrm{so}$-regular.

Conversely, let $\C$ be $\F_\mathrm{bo}$-regular and ff-regular; without loss of generality we assume it is also small.  Since fully faithful regular epis are effective $\F_\mathrm{so}$-quotients (by Lemma~\ref{lemma:eqrel-lemma}) and stable (by assumption), the topology on $\C$ generated by these maps together with the $\F_\mathrm{bo}$-regular topology is subcanonical; thus we have a full embedding $J \colon \C \to \cat{Sh}(\C)$. As the given topology contains the $\F_\mathrm{bo}$-regular one, the argument of Theorem~\ref{thm:bo-char} shows that $J$ preserves quotients of $\F_\mathrm{bo}$-kernels, whence $\C$ is closed in $\cat{Sh}(\C)$ under $\F_\mathrm{bo}$-kernel--quotient factorisations; moreover, $J$ sends fully faithful regular epis to fully faithful $\F_\mathrm{so}$-quotients, thus to fully faithful regular epis, and so adapting the argument of Proposition~\ref{prop:preserve-quotients-preserve-colims}, we conclude that $\C$ is closed in $\cat{Sh}(\C)$ under coequalisers of  kernel-pairs of fully faithful maps. To show that $\C$ is $\F_\mathrm{so}$-regular, it now suffices to show that $\C$ is closed in $\cat{Sh}(\C)$ under $\F_\mathrm{so}$-kernel--quotient factorisations. So let $f \colon A \to D$ in $\C$; we factorise it in $\cat{Sh}(\C)$ as
\[
\cd[@R-2em@C-0.7em]{
f = A \ar[rr]^-{g} \ar[dr]_{k} & & C  \ar[rr]^-{h} &  & D \\
& B \ar[ur]_\ell
}
\]
where
$f = h g$ is an $\F_\mathrm{so}$-kernel--quotient factorisation, and $g = \ell k$ is an $\F_\mathrm{bo}$-kernel--quotient factorisation. We must show that $C \in \C$. Now $h$ is full monic since $\cat{Sh}(\C)$ is $\F_\mathrm{so}$-regular, whence $f = (h \ell) k$ is also an $\F_\mathrm{bo}$-kernel--quotient factorisation, and so $B \in \C$ since $A$ and $D$ are. Next, $\ell$ is fully faithful since $\cat{Sh}(\C)$ is $\F_\mathrm{bo}$-regular; it is also acute, since $g$ is, whence $\ell$ is an effective $\F_\mathrm{so}$-quotient. Thus, by Lemma~\ref{lemma:eqrel-lemma}, $\ell$ is a regular epi; it is thus the coequaliser of its own kernel-pair. But since $h$ is monic, this is equally the coequaliser of the kernel-pair of $h\ell$; and as $h\ell$ is fully faithful, and $\C$ is closed in $\cat{Sh}(\C)$ under coequalisers of kernel-pairs of fully faithful maps, we conclude that $C \in \C$ since $B$ and $D$ are.
\end{proof}

We now turn to the relationship between $\F_\mathrm{so}$ and $\F_\mathrm{bo}$-exactness; and here we obtain a slightly tighter correspondence. Let us call a $2$-category \emph{ff-exact} if it is ff-regular, and every fully faithful equivalence relation (one whose source and target maps are fully faithful) admits a fully faithful coequaliser and is effective.
\begin{Prop}\label{prop:soexact-boexact}
A $2$-category is $\F_\mathrm{so}$-exact if and only if it is $\F_\mathrm{bo}$-exact and ff-exact.
\end{Prop}
\begin{proof}
We first prove that if $\C$ is $\F_\mathrm{so}$-exact, then it is ff-exact. Certainly it is ff-regular by the preceding result; moreover, given a fully faithful equivalence relation $E \rightrightarrows A$ in $\C$, we may by Lemma~\ref{lemma:eqrel-lemma} form its coequaliser as the $\F_\mathrm{so}$-quotient of the associated $\F_\mathrm{so}$-congruence~\eqref{eq:ff-kernel}; and effectivity of this $\F_\mathrm{so}$-congruence says in particular that $E$ is the kernel-pair of its quotient, and so an effective equivalence relation.

We now show that any $\F_\mathrm{so}$-exact $\C$ is $\F_\mathrm{bo}$-exact. Without loss of generality, we take $\C$ to be small; equipping it with the $\F_\mathrm{so}$-regular topology, we obtain an $\F_\mathrm{so}$-exact embedding $\C \to \cat{Sh}(\C)$, and it suffices to show that $\C$ is closed in $\cat{Sh}(\C)$ under quotients of $\F_\mathrm{bo}$-congruences. 
Let $X \in [\K_\mathrm{bo},  \C]$ be an $\F_\mathrm{bo}$-congruence, and form its quotient $q \colon X1 \to Q$ in $\cat{Sh}(\C)$; we must show that $Q$ actually lies in $\C$. Since $q$ is acute, it is an effective $\F_\mathrm{so}$-quotient map in $\cat{Sh}(\C)$, and so is the $\F_\mathrm{so}$-quotient of its $\F_\mathrm{so}$-kernel $V$; thus, if we can show that each vertex of $V$ lies in $\C$, it will then follow that $q$ does too, since $\C$ is closed in $\cat{Sh}(\C)$ under quotients of $\F_\mathrm{so}$-congruences. 
Now, the underlying $\F_\mathrm{bo}$-congruence of $V$ is the $\F_\mathrm{bo}$-kernel of $q$: but this is simply $X$, since $\F_\mathrm{bo}$-congruences are effective in $\cat{Sh}(\C)$, and thus $V1$, $V2$ and $V3$ all lie in $\C$; it remains to show the same for $V2'$. But we have that
\[
Vi = V1 \xrightarrow{} V2' \xrightarrow{V j} V 2
\]
in $\cat{Sh}(\C)$, where the first component is the map witnessing the relation $V2' \rightrightarrows V1$ as reflexive. This map is equally the diagonal $X1 \to X1 \times_Q X1$ of $q$'s kernel-pair, and thus acute by Proposition~\ref{prop:codescent-in-boregular}, since $q$ is an effective $\F_\mathrm{bo}$-quotient map; on the other hand, $Vj$ is full monic because $V$ is an $\F_\mathrm{so}$-congruence. Thus the above is an $\F_\mathrm{so}$-kernel--quotient factorisation in $\cat{Sh}(\C)$, and so $V2'$ lies in $\C$ since $V1$ and $V2$ do.

Finally, we show that if $\C$ is $\F_\mathrm{bo}$-exact and ff-exact, then it is $\F_\mathrm{so}$-exact. Without loss of generality, we assume $\C$ is small; now by arguing as in the proof of the preceding proposition, we can find an $\F_\mathrm{bo}$-exact and ff-exact embedding $\C \to \cat{Sh}(\C)$, and to complete the proof, it will suffice to show that $\C$ is closed in $\cat{Sh}(\C)$ under quotients of $\F_\mathrm{so}$-congruences. So let $X \in [\K_\mathrm{so}, \C]$ be an $\F_\mathrm{so}$-congruence, and form its $\F_\mathrm{so}$-quotient $q \colon X1 \to Q$ in $\cat{Sh}(\C)$; we must show that $Q$ lies in $\C$. Take an $\F_\mathrm{bo}$-kernel--quotient factorisation $q = kh \colon X1 \to P \to Q$; note that the $\F_\mathrm{bo}$-kernel of $q$ is the underlying $\F_\mathrm{bo}$-congruence of $X$, whence $P$ is the $\F_\mathrm{bo}$-quotient of this congruence and thus lies in $\C$. Now $k$ is acute, since $q$ is, and hence an effective $\F_\mathrm{so}$-quotient; it is also fully faithful by $\F_\mathrm{bo}$-exactness of $\cat{Sh}(\C)$, and hence is a fully faithful regular epi. It thus suffices to show that the fully faithful kernel-pair $P \times_Q P \rightrightarrows P$ of $k$ lies in $\C$, as then $Q$, its coequaliser, will too. So consider the serially commuting diagram of kernel-pairs
\[
\cd{
X1 \times_Q X1 \ar@<3pt>[r]^-{s'} \ar@<-3pt>[r]_-{t'} \ar[d]_{h \times_Q h} & X1 \ar[d]^h \\
P \times_Q P \ar@<3pt>[r]^-{s} \ar@<-3pt>[r]_-{t} & P \rlap{ .}
}
\]
We noted earlier that $P \in \C$; but as $q$ is the $\F_\mathrm{so}$-quotient of the effective $\F_\mathrm{so}$-congruence $X$, we have $X1 \times_Q X1 \cong X2'$ in $\C$ too. Now $h \times_Q h = h \times_Q 1 \circ 1 \times_Q h$ is a composite of $\F_\mathrm{bo}$-quotient maps in $\cat{Sh}(\C)$, and so an $\F_\mathrm{bo}$-quotient; on the other hand, $s$ is fully faithful since $k$ is, and so $s.h \times_Q h$ is an $\F_\mathrm{bo}$-kernel--quotient factorisation of $hs'$; thus $P \times_Q P$ lies in $\C$, since $P$ and $X1 \times_Q X1$ do.
\end{proof}

Finally, we consider the relationship with $\F_\mathrm{bof}$-regularity.

\begin{Prop}\label{prop:boe-sor-bofr}
Any $\F_\mathrm{bo}$-exact and $\F_\mathrm{so}$-regular $2$-category is $\F_\mathrm{bof}$-regular.
\end{Prop}
\begin{proof}
Let $\C$ be $\F_\mathrm{bo}$-exact and $\F_\mathrm{so}$-regular; without loss of generality, we assume it is also small. On taking sheaves on $\C$ for its $\F_\mathrm{so}$-regular topology, we obtain an  embedding $\C \to \cat{Sh}(\C)$ into a $2$-topos; the embedding is clearly $\F_\mathrm{so}$-regular, but in fact also $\F_\mathrm{bo}$-exact, by the argument of Theorem~\ref{thm:boex-char}, since the $\F_\mathrm{so}$-regular topology contains the $\F_\mathrm{bo}$-regular one. To show that $\C$ is $\F_\mathrm{bof}$-regular, it now suffices to show that $\C$ is closed in the $2$-topos $\cat{Sh}(\C)$ under $\F_\mathrm{bof}$-kernel--quotient factorisations. 
Given a morphism $f \colon C \to D$ in $\C$, let $f = me$ be its $\F_\mathrm{bof}$-kernel--quotient factorisation in $\cat{Sh}(\C)$, and let $X$ be the $\K_\mathrm{bo}$-kernel of $e$. Since $e$ is an $\F_\mathrm{bo}$-strong epi, it is an effective $\F_\mathrm{bo}$-quotient map in the $2$-topos $\cat{Sh}(\C)$, and so the $\F_\mathrm{bo}$-quotient of $X$ is again $e$; arguing as in the previous proof, it now suffices to show that each vertex of $X$ lies in $\C$, as then the $\F_\mathrm{bo}$-quotient $e$ will do so too. Clearly $X1 = C$ lies in $\C$; as for $X2 = e \mathord \mid e$, consider the factorisation
\[
\gamma_f = C^\mathbf 2 \xrightarrow{\gamma_e} e \mathord \mid e \xrightarrow{m \mathord \mid m} f \mathord \mid f
\]
in $\cat{Sh}(\C)$. Because $m$ is faithful, $m \mathord \mid m$ is easily seen to be full monic; on the other hand, $\gamma_e$ is acute by Proposition~\ref{prop:coeq-in-bofregular}, since $e$ is an $\F_\mathrm{bof}$-quotient map in $\cat{Sh}(\C)$. Thus the above is an $\F_\mathrm{so}$-kernel--quotient factorisation in $\cat{Sh}(\C)$, and so $X2 = e \mathord \mid e$ lies in $\C$ since $C^\mathbf 2$ and $f \mathord \mid f$ do. Finally, $X3$ lies in $\C$ as it is the pullback $X2 \times_{X1} X2$.
\end{proof}

\section{Examples}\label{sec:examples}
In this final section, we exhibit various classes of $2$-categories as instances of our two-dimensional regularity and exactness notions.

\subsection{$2$-toposes}\label{subsec:2topos}
Of course, by Proposition~\ref{prop:v-topos-phi-exact} every $2$-topos is $\F_\mathrm{so}$-exact, $\F_\mathrm{bo}$-exact and $\F_\mathrm{bof}$-exact: a fact we have used extensively in the preceding sections. $\cat{Cat}$ and every presheaf $2$-category $[\C^\op, \cat{Cat}]$ are $2$-toposes; more generally, if $\E$ is any Grothendieck topos, then $\cat{Cat}(\E)$, the $2$-category of categories internal to $\E$, is a $2$-topos; for indeed, if $\E$ is reflective in $[\D^\op, \cat{Set}]$ via a left-exact reflector, then $\cat{Cat}(\E)$ is reflective in $\cat{Cat}([\D^\op, \cat{Set}]) \cong [\D^\op, \cat{Cat}]$ via a left-exact reflector, and hence a $2$-topos. 

\subsection{Locally discrete $2$-categories}
A $2$-category is called \emph{locally discrete} if its only $2$-cells are identities.
\begin{Prop}\label{prop:discrete-exactness}
Let $\C$ be a finitely complete, locally discrete $2$-category, and let $\C_0$ be its underlying ordinary category. Then:
\begin{enumerate}[(a)]
\item $\C$ is $\F_\mathrm{so}$-regular if and only if $\C_0$ is regular;
\item $\C$ is $\F_\mathrm{bo}$-regular or $\F_\mathrm{bo}$-exact if and only if $\C_0$ is a preorder;
\item $\C$ is $\F_\mathrm{so}$-exact if and only if $\C_0$ is a preorder;
\item $\C$ is always $\F_\mathrm{bof}$-regular and $\F_\mathrm{bof}$-exact.
\end{enumerate}
\end{Prop}
\begin{proof}
For (a), observe that the $\F_\mathrm{so}$-kernel of an arrow $f$ in the locally discrete $\C$ is the extended kernel-pair diagram
\[
\cd{
& A \times_B A \ar@{=}[d]
\\
A \times_B A \times_B A \ar@<6pt>[r]^-{p} \ar[r]|-{m} \ar@<-6pt>[r]_-{q} &
A \times_B A \ar@<6pt>[r]^-{d} \ar@{<-}[r]|-{i} \ar@<-6pt>[r]_-{c} &
A\rlap{ ,}
}
\]
whilst the $\F_\mathrm{so}$-quotient of $X \in [\K_\mathrm{so}, \C]$ is simply the coequaliser of $(Xd, Xc) \colon X2 \rightrightarrows X1$. It follows that $\C$ admits $\F_\mathrm{so}$-quotients of $\F_\mathrm{so}$-kernels just when $\C_0$ admits coequalisers of kernel-pairs. In this situation, the $\F_\mathrm{so}$-quotient maps are the regular epimorphisms; as such they are always effective, and will be stable under pullback in $\C$ just when they are so in $\C_0$; whence $\C$ is $\F_\mathrm{so}$-regular just when $\C_0$ is regular.

For (b), note that $\F_\mathrm{bo}$-kernels in $\C$, like $\F_\mathrm{so}$-kernels, are just extended kernel-pair diagrams, and $\F_\mathrm{bo}$-quotients, just coequalisers; it follows that the $\F_\mathrm{bo}$-quotient maps, like the $\F_\mathrm{so}$-quotient maps, are the regular epimorphisms. Now if $\C$ is $\F_\mathrm{bo}$-regular, then for every $\F_\mathrm{bo}$-quotient map $f \colon A \to B$, the diagonal $\delta_f \colon A \to A \times_B A$ is also an $\F_\mathrm{bo}$-quotient map, hence regular epi. Since $\delta_f$ is always monic, it is thus invertible: which is to say that $f$ itself is monic. Since $f$ is also regular epi, it must be invertible: and thus the $\F_\mathrm{bo}$-quotient maps in $\C$ are precisely the isomorphisms. It follows that every map of $\C$ is fully faithful; thus every map of $\C_0$ is monic, which, in combination with a terminal object, forces $\C_0$ to be a preorder. Conversely, if $\C_0$ is a preorder, then the $\F_\mathrm{bo}$-quotient maps are simply the isomorphisms, and so satisfy the requisite stability properties for $\C$ to be $\F_\mathrm{bo}$-regular. Moreover, all $\F_\mathrm{bo}$-congruences in $\C$ are trivial, and so admit effective $\F_\mathrm{bo}$-quotients; whence $\C$ is $\F_\mathrm{bo}$-exact.

For (c), if $\C$ is $\F_\mathrm{so}$-exact, then it is $\F_\mathrm{bo}$-exact by Proposition~\ref{prop:soexact-boexact}, whence $\C_0$ is a preorder by (b); conversely, if 
$\C_0$ is a preorder, then it is regular, so $\C$ is $\F_\mathrm{so}$-regular. Furthermore, the only internal categories in $\C$ are the trivial ones; whence the only $\F_\mathrm{so}$-congruences are trivial, and as such, are effective: so $\C$ is $\F_\mathrm{so}$-exact.

Finally, for (d), note that, since any two parallel two cells in $\C$ are equal, all $\F_\mathrm{bof}$-quotients exist, and the $\F_\mathrm{bof}$-quotient maps are the isomorphisms. It follows immediately that $\C$ satisfies the elementary conditions characterisating $\F_\mathrm{bof}$-regularity. Finally, the only $\F_\mathrm{bof}$-congruences are again trivial, and as such are easily effective; whence $\C$ is also $\F_\mathrm{bof}$-exact.
\end{proof}

%

\subsection{Models of finite product theories}\label{subsec:finite-product-theories}
Because $\F_\mathrm{so}$-, $\F_\mathrm{bo}$- and $\F_\mathrm{bof}$-quotient maps are stable under finite products in $\cat{Cat}$, we may apply Proposition~\ref{prop:sifted} to obtain:
\begin{Prop}\label{prop:fpexact}
Let $\A$ be a small $2$-category with finite products.
If the $2$-category $\C$ is $\F_\mathrm{so}$-, $\F_\mathrm{bo}$- or $\F_\mathrm{bof}$-regular or exact, then so is $\cat{FP}(\A, \C)$.
\end{Prop}
The scope of this result is quite considerable, as the following result indicates.
\begin{Prop}\label{prop:examples-fp-theories}
Consider any of the following notions: strict monoidal categories; monoidal categories; braided or symmetric monoidal categories; categories with finite products; categories with finite coproducts; distributive categories; pointed categories; categories with a zero object; categories equipped with a monad; bicategories with a fixed object set; $\cat{Cat}$-operads; pseudo-$\cat{Cat}$-operads (in which composition is only associative up to coherent $2$-cells); (pseudo-)$\cat{Cat}$-multicategories with a fixed object set; monoidal globular categories in the sense of~\cite{Batanin1998Monoidal}; pairs $(A,B)$ of a monoidal category $A$ with a lax action on a category $B$. 
In each case, the $2$-category whose objects are instances of that notion and whose morphisms are strict structure-preserving maps is $\F_\mathrm{so}$-, $\F_\mathrm{bo}$- and $\F_\mathrm{bof}$-exact.
More generally, if a $2$-category $\C$ possesses any one of our regularity or exactness properties, then so too does the $2$-category of instances of any of the above notions in $\C$ with strict structure-preserving maps; thus, for example, the $2$-category of internal monoidal categories and strict monoidal internal functors in a Grothendieck topos is $\F_\mathrm{so}$-, $\F_\mathrm{bo}$- and $\F_\mathrm{bof}$-exact.
\end{Prop}
\begin{proof}
Take, for instance, the case of categories equipped with finite coproducts. Let $\A$ be the free $2$-category with finite products generated by an object $X$, morphisms $c_n \colon X^n \to X$ for each $n$, and $2$-cells $\eta_n$ and $\varepsilon_n$ witnessing that $c_n$ is left adjoint to the diagonal $X \to X^n$. Such an object $\A$ may be constructed as in the proof of Proposition~\ref{prop:ker-quot-lex-weights} from bicolimits of frees in $\cat{2}\text-\cat{Lex}$. Now $\cat{FP}(\A, \cat{Cat})$ is $2$-equivalent to the $2$-category of categories with finite coproducts and strict structure-preserving maps, whilst $\cat{FP}(\A, \C)$ is the $2$-category of objects with internal finite coproducts in $\C$ and strict maps. Similar arguments pertain for each of the other notions.
\end{proof}

\begin{Rk}
The restriction to \emph{strict} structure-preserving maps in this result is necessary, since the corresponding $2$-categories whose maps preserve the structure only up to isomorphism typically do not possess all finite $2$-categorical limits, but only finite pie limits in the sense of~\cite{Bird1989Flexible}; and without all finite limits, we cannot obtain the stability under \emph{strict} pullbacks of the various classes of maps required for the material of~\cite{Street1982Two-dimensional} to be applicable. To describe the exactness of $2$-categories of pseudomorphisms thus requires a \emph{bicategorical}\footnote{Another possibility would be to work with the $\F$-categories of~\cite{Lack2011Enhanced}.} analogue of the  theory, which, as indicated in the introduction, is outside the scope of this paper.

The strict structure-preserving maps are commonly held to be only of theoretical importance, which may appear to  limit severely the usefulness of our results. However, the restriction to strict maps appears entirely natural if we understand the structured objects under consideration as \emph{theories} rather than semantic domains. For example, the $2$-category $\cat{SymMonCat}_s$ of symmetric monoidal categories and strict maps can be understood as a $2$-category of generalised PROPs~\cite{Mac-Lane1965Categorical}, and its $\F_\mathrm{bof}$-, $\F_\mathrm{bo}$- and $\F_\mathrm{so}$-exactness now accounts for the possibility of extending a PROP by adding new equations between operations, new operations, and new equations between sorts. 
\end{Rk}
\begin{Rk}
The instances in $\cat{Cat}$ of each of the structures listed in the preceding result can also be captured as algebras for a \emph{strongly finitary} $2$-monad $\mathsf T$ on a $2$-category of the form $\cat{Cat}^X$ for some set $X$, and indeed the monads approach is more commonly used in describing such structure borne by categories. On the other hand the 2-category of algebras and strict algebra morphisms $\mathsf T\text-\cat{Alg}_s$ associated to such a 2-monad is equivalent to $\cat{FP}(\A, \cat{Cat})$ for a small 2-category $\A$ admitting finite products. This follows from the fact that $\cat{Cat}$ is a locally strongly finitely presentable 2-category---as proven, for instance, in~\cite[Proposition~8.31]{Bourke2010Codescent}---and the results of \cite{Lack2011Notions} which describe the correspondence between such 2-monads and finite product theories in a general enriched setting. Via this equivalence and Proposition~\ref{prop:fpexact} we conclude that the 2-category of algebras for a strongly finitary 2-monad on $\cat{Cat}^X$ is exact in each of our senses.
\end{Rk}

\subsection{Internal categories}
Let $\E$ be a finitely complete category, viewed as a locally discrete $2$-category. We write $\cat{Cat}(\E)$ for the $2$-category of internal categories, internal functors and internal natural transformations in $\E$, and $\Delta$ for the embedding $2$-functor $\E \to \cat{Cat}(\E)$ sending $X$ to the discrete internal category on $X$. Note that $\cat{Cat}(\E)$ is finitely complete as a $2$-category.

\begin{Prop}
For any finitely complete $\E$, the $2$-category $\cat{Cat}(\E)$ is $\F_\mathrm{bo}$-exact.
\end{Prop}
\begin{proof} If $\E$ is small, so too is $\cat{Cat}(\E)$; whence it suffices by Proposition~\ref{prop:phi-small-reduction} to prove the result when $\E$ is small.
The following observations are easily verified:
\begin{enumerate}[(a)]
\item  To give a pointwise discrete $\F_\mathrm{bo}$-congruence 
\begin{equation}\label{eq:pointwise-disc}
\cd{
\Delta X_2 \ar@<6pt>[r]^{} \ar[r]|{} \ar@<-6pt>[r]_{} &
\Delta X_1 \ar@<6pt>[r]^{} \ar@{<-}[r]|{} \ar@<-6pt>[r]_{} &
\Delta X_0 
}
\end{equation}
in $\cat{Cat}(\E)$ is to give the truncated nerve of an internal category $\cat X$ in $\E$; every such congruence admits an $\F_\mathrm{bo}$-quotient given by 
$\cat X$ together with the identity-on-objects internal functor $\Delta X_0 \to \cat X$.
\vskip0.5\baselineskip
\item For any $A \in \E$, $\cat{Cat}(\E)(\Delta A, \thg) \colon \cat{Cat}(\E) \to \cat{Cat}$ preserves $\F_\mathrm{bo}$-quotients of pointwise discrete $\F_\mathrm{bo}$-congruences as in (a). \vskip0.5\baselineskip
\item If $X, Y \in [\K , \cat{Cat}]$ are $\F_\mathrm{bo}$-congruences, and $\phi \colon X \to Y$ a pointwise bijective-on-objects transformation between them, then the induced map $Q \phi \colon QX \to QY$ on $\F_\mathrm{bo}$-quotients is invertible.
\end{enumerate}
From (a) and (b), we deduce by~\cite[Theorem~5.19(iv)]{Kelly1982Basic} that the singular functor $J = \cat{Cat}(\E)(\Delta, 1) \colon \cat{Cat}(\E) \to [\E^\op, \cat{Cat}]$ is fully faithful. $[\E^\op, \cat{Cat}]$ is $\F_\mathrm{bo}$-exact, since $\cat{Cat}$ is, and so it suffices by Theorem~\ref{thm:phi-exact-embedding} to show that $\cat{Cat}(\E)$ is closed in $[\E^\op, \cat{Cat}]$ under finite limits and $\F_\mathrm{bo}$-quotients of $\F_\mathrm{bo}$-congruences. Closure under finite limits is clear, since $\cat{Cat}(\E)$ has these and $J$ preserves them. On the other hand, given a congruence $X \in [\K_\mathrm{bo}, \cat{Cat}(\E)]$, we have a pointwise internally bijective-on-objects map $\phi \colon \Delta(X\thg)_0 \to X$ from a pointwise discrete congruence. It is evident that $J\phi$ is pointwise bijective-on-objects in $[\E^\op, \cat{Cat}]$, whence by (c), the $\F_\mathrm{bo}$-quotients of $J\Delta(X\thg)_0$ and $JX$ are isomorphic. But by (b), the former quotient lies in the essential image of $J$; whence the latter does too.
\end{proof}
\begin{Cor}
$\Delta \colon \E \to \cat{Cat}(\E)$ exhibits $\cat{Cat}(\E)$ as the $\F_\mathrm{bo}$-exact completion of $\E$.
\end{Cor}
\begin{proof}
It suffices to show that the replete image of the fully faithful singular functor $J \colon \cat{Cat}(\E) \to [\E^\op, \cat{Cat}]$ is $\Phi_{\F_\mathrm{bo}}^\mathrm{ex}(\E)$. But the preceding proof shows that this replete image contains the representables and is closed under quotients of $\F_\mathrm{bo}$-congruences, and moreover, that every object in it is an $\F_\mathrm{bo}$-quotient of a pointwise representable $\F_\mathrm{bo}$-congruence.
\end{proof}

\begin{Prop}
If $\E$ is a regular $1$-category, then $\cat{Cat}(\E)$ is $\F_\mathrm{so}$-regular and $\F_\mathrm{bof}$-regular.
\end{Prop}
\begin{proof}
It suffices to prove that $\cat{Cat}(\E)$ is ff-regular, as then $\F_\mathrm{so}$-regularity follows from Proposition~\ref{prop:reg-ff} and $\F_\mathrm{bof}$-regularity  from Proposition~\ref{prop:boe-sor-bofr}.
 Given $f \colon \cat X \to \cat Y$ fully faithful in $\cat{Cat}(\E)$, we factorise it as on the left in
\[
\cd[@C+0.5em]{
X_1 \ar@{->>}[r]^{e_1} \ar@<-3pt>[d]_{d_X} \ar@<3pt>[d]^{c_X} &
Z_1 \ar@{ (->}[r]^{m_1}  \ar@<-3pt>[d]_{d_Z} \ar@<3pt>[d]^{c_Z} &
Y_1 \ar@<-3pt>[d]_{d_Y} \ar@<3pt>[d]^{c_Y} \\
X_0 \ar@{->>}[r]_-{e_0}  &
Z_0 \ar@{ (->}[r]_-{m_0}   &
Y_0
} \qquad
\cd[@C+0.2em]{
X_1 \ar[d]_{(d_X, c_X)} \ar[r]^{e_1} &
Z_1 \ar[d]_{(d_Z, c_Z)} \pushoutcorner \ar[r]^{m_1} & Y_1 \ar[d]^{(d_Y, c_Y)} \\
X_0 \times X_0 \ar[r]_{e_0 \times e_0} & Z_0 \times Z_0 \ar[r]_{m_0 \times m_0} & Y_0 \times Y_0\rlap{ ,}
}
\]
where $f_0 = e_0m_0$ is a (regular epi, mono) factorisation in $\E_0$, and $Z_1$, $d_Z$ and $c_Z$ are obtained via a pullback as in the right-hand rectangle of the second diagram above; this method of definition easily implies that $e$ and $m$ are internal functors. Note that, as the large rectangle in this diagram is also a pullback, the left-hand one is too; whence $e_1$ is regular epi, and $e \colon \cat X \to \cat Z$ is fully faithful. Now $e$ is the pointwise coequaliser of the kernel pair of $f$, and so \emph{a fortiori} the coequaliser in $\cat{Cat}(\E)$. Thus $\cat{Cat}(\E)$ admits fully faithful coequalisers of kernel-pairs of fully faithful morphisms; stability under pullback follows from the  stability of regular epis in $\E_0$.
\end{proof}

In the situation of the preceding proposition, we may identify explicitly the $\F_\mathrm{so}$-quotients and the $\F_\mathrm{bof}$-quotients in $\cat{Cat}(\E)$. It is easy to see that if a $1$-category $\E$ with finite limits admits a factorisation system $(\L, \R)$, then the $2$-category $\cat{Cat}(\E)$ admits two factorisation systems ($\L$ on objects, $\R$ on objects and fully faithful) and (bijective on objects and $\L$ on morphisms, locally $\R$). In particular, if $\E$ is a regular $1$-category then it admits the factorisation system (regular epi, mono), and now the two induced factorisation systems on $\cat{Cat}(\E)$ have as their corresponding right classes the $\F_\mathrm{so}$-monics and $\F_\mathrm{bof}$-monics respectively. It follows that the $\F_\mathrm{so}$-quotients (=acute maps) in $\cat{Cat}(\E)$ are those which are regular epi on objects, and that the $\F_\mathrm{bof}$-quotients are those which are bijective on objects and regular epi on morphisms.

\begin{Prop}
If $\E$ is a Barr-exact $1$-category, then $\cat{Cat}(\E)$ is $\F_\mathrm{so}$-exact and $\F_\mathrm{bof}$-exact.
\end{Prop}
\begin{proof}
To prove $\F_\mathrm{so}$-exactness of $\cat{Cat}(\E)$, it suffices by Proposition~\ref{prop:soexact-boexact} to show ff-exactness.
Given $(s,t) \colon \cat E \to \cat X$ a fully faithful equivalence relation in $\cat{Cat}(\E)$, we obtain on taking nerves an equivalence relation $(Ns, Nt) \colon N\cat E \to N \cat X$ in $[\Delta^\op, \E]$. Since $\E_0$ is Barr-exact, $(Ns, Nt)$ admits a pointwise coequaliser $q \colon N\cat X \to Q$ of which it is the kernel-pair. We claim that in fact $Q$ is (isomorphic to) the nerve of an internal category $\cat Q$.
To show this, we must check that the Segal maps $Q_n \to Q_1 \times_{Q_0} \cdots \times_{Q_0} Q_1$ are invertible. We illustrate with the case $n = 2$; the argument for higher $n$ is identical. Consider first the diagram on the left in:
\[
\cd[@C+0.2em]{
E_1 \ar[d]_{(d, c)} \ar@<3pt>[r]^{s_1} \ar@<-3pt>[r]_{t_1} &
X_1 \ar[d]|{(d, c)} \ar@{->>}[r]^{q_1} \ar@{}[dr]|{\text{(1)}} & Q_1 \ar[d]^{(d, c)} \\
{E_0}^2 \ar@<3pt>[r]^{{s_0}^2} \ar@<-3pt>[r]_{{t_0}^2} & {X_0}^2 \ar@{->>}[r]_{{q_0}^2} & {Q_0}^2
} \qquad 
\cd[@C+1em]{
X_1 \times_{X_0} X_1 \ar@{}[dr]|{\text{(2)}} \ar[d]|{(d \pi_1, c \pi_1, c \pi_2)} \ar@{->>}[r]^{q_1 \times_{q_0} q_1} & Q_1 \times_{Q_0} Q_1 \ar[d]|{(d \pi_1, c \pi_1, c \pi_2)} \\
{X_0}^3 \ar@{->>}[r]_{{q_0}^3} & {Q_0}^3
}
\]
Both rows are exact forks (the bottom since regular epis are stable under products in $\E_0$) and both left-hand squares are pullbacks, since $s$ and $t$ are fully faithful; thus, by~\cite[Proposition~4.2]{Grothendieck1960Techniques}, the square (1) is also a pullback. It follows that the square (2) is a pullback. Similarly, in the diagram on the left of:
\[
\cd[@C+0.2em]{
E_2 \ar[d]_{(dp, cp, cq)} \ar@<3pt>[r]^{s_2} \ar@<-3pt>[r]_{t_2} &
X_2 \ar[d]|{(dp, cp, cq)} \ar@{->>}[r]^{q_2} \ar@{}[dr]|{\text{(3)}} & Q_2 \ar[d]^{(dp, cp, cq)} \\
{E_0}^3 \ar@<3pt>[r]^{{s_0}^3} \ar@<-3pt>[r]_{{t_0}^3} & {X_0}^3 \ar@{->>}[r]_{{q_0}^3} & {Q_0}^3
} \qquad 
\cd[@C+1em]{
X_2 \ar[r]^{q_2}  \ar[d]_{(p, q)} \ar@{}[dr]|{\text{(4)}} & Q_2 \ar[d]^{(p,q)} \\
X_1 \times_{X_0} X_1 \ar@{->>}[r]_-{q_1 \times_{q_0} q_1} & Q_1 \times_{Q_0} Q_1 \rlap{ .}
}
\]
both rows are exact forks, and both left-hand squares pullbacks, whence the square (3) is also a pullback. Finally, the square (4) is a pullback, since pasting with the pullback (2) yields the pullback (3). But the left-hand arrow of (4) is invertible, and since pullback along a stable regular epimorphism is conservative, we conclude that the right-hand arrow, the Segal map $Q_2 \to Q_1 \times_{Q_0} Q_1$, is invertible as required. The invertibility of the higher Segal maps follows similarly; thus $Q$ is the nerve of an internal category $\cat Q$. It follows that $q$ is an internal functor $\cat X \to \cat Q$, fully faithful since (1) is a pullback, and clearly an effective coequaliser of $(s,t)$ in $\cat{Cat}(\E)$.
Thus $\cat{Cat}(\E)$ is ff-exact, and hence $\F_\mathrm{so}$-exact. 

It remains to prove $\F_\mathrm{bof}$-exactness of $\cat{Cat}(\E)$. Let $\cat{Sh}(\E)$ be the $1$-topos of sheaves on $\E_0$ for the regular topology; thus we have a full embedding $\E \to \cat{Sh}(\E)$ preserving finite limits and regular epis, inducing a full embedding $\cat{Cat}(\E) \to \cat{Cat}(\cat{Sh}(\E))$ preserving finite limits and acute maps; it follows that $\cat{Cat}(\E)$ is closed in the $2$-topos $\cat{Cat}(\cat{Sh}(\E))$ under quotients of $\F_\mathrm{so}$- and $\F_\mathrm{bo}$-congruences. Since $\cat{Cat}(\cat{Sh}(\E))$ is moreover $\F_\mathrm{bof}$-exact, to complete the proof, it suffices to show that $\cat{Cat}(\E)$ is closed in $\cat{Cat}(\cat{Sh}(\E))$ under quotients of $\F_\mathrm{bof}$-congruences.

So given a congruence $X \in [\K_\mathrm{bof}, \cat{Cat}(\E)]$, form its $\F_\mathrm{bof}$-quotient $q \colon X1 \to Q$ in $\cat{Cat}(\cat{Sh}(\E))$, and the discrete cover $\phi \colon \Delta(X1)_0 \to X1$ of $X1$. The composite $q \phi \colon \Delta(X1)_0 \to Q$ is an effective $\F_\mathrm{bo}$-quotient, since both components are, and hence is the $\F_\mathrm{bo}$-quotient of its own $\F_\mathrm{bo}$-kernel $V$. Since $\cat{Cat}(\E)$ is closed in $\cat{Cat}(\cat{Sh}(\E))$ under quotients of $\F_\mathrm{bo}$-congruences, it now suffices to prove that $V$ lies in $\cat{Cat}(\E)$. Clearly $V1 = \Delta(X1)_0$ does; and $V3$ will do so as soon as $V2$ does; thus it suffices to show that $V2 = \Delta(q |q)_0$ lies in $\cat{Cat}(\E)$. Consider the diagram
\[
\cd[@-0.4em]{
\Delta(X1^\mathbf 2)_0 \ar[r]^{\Delta(\gamma_q)_0} \ar[d]_{\phi} & \Delta(q | q)_0 \ar[d]^\phi \\
X1^\mathbf 2 \ar[r]_{\gamma_q} & q | q
}
\]
in $\cat{Cat}(\cat{Sh}(\E))$. By Proposition~\ref{prop:coeq-in-bofregular}, $\gamma_q$ is acute since $q$ is a coequifier, and thus its object map is a regular epimorphism; thus the object map of $\Delta(\gamma_q)_0$ is also regular epi, whence $\Delta(\gamma_q)_0$ is acute. It is thus the quotient of its own $\F_\mathrm{so}$-kernel $W$ in $\cat{Cat}(\cat{Sh}(\E))$; since $\cat{Cat}(\E)$ is closed in $\cat{Cat}(\cat{Sh}(\E))$ under quotients of $\F_\mathrm{so}$-congruences, it now suffices to show that $W$ lies in $\cat{Cat}(\E)$. Clearly $W1 = \Delta(X1^\cat 2)_0$ is in $\cat{Cat}(\E)$; and it is easy to see that $W2 = W2' \cong \Delta(X2)_0$, so that $W2$ and $W2'$ lie in $\cat{Cat}(\E)$; whence finally $W3 = W2 \times_{W1} W2$ lies in $\cat{Cat}(\E)$, as required.
\end{proof}

Finally, combining the results of this section with those of Section~\ref{subsec:finite-product-theories}, we obtain:
\begin{Prop}
Let $\E$ be a finitely complete category. For each of the notions listed in Proposition~\ref{prop:examples-fp-theories}, the $2$-category of internal instances of that notion in $\cat{Cat}(\E)$---with strict structure-preserving maps---is $\F_\mathrm{bo}$-exact. It is moreover $\F_\mathrm{so}$- and $\F_\mathrm{bof}$-regular if $\E$ is regular, and $\F_\mathrm{so}$- and $\F_\mathrm{bof}$-exact if $\E$ is Barr-exact.
\end{Prop}

\bibliographystyle{acm}

\bibliography{bibdata}

\end{document}